\documentclass[11pt, reqno]{amsart}
\usepackage{latexsym,amsmath,amssymb, amsthm, mathscinet}
\usepackage{cases, verbatim, esint, todonotes}
\usepackage[active]{srcltx}
\usepackage[colorlinks]{hyperref}
\usepackage{xcolor}

\usepackage[margin=1.5in]{geometry}

\newtheorem{thm}{Theorem}
\newtheorem{lem}[thm]{Lemma}
\newtheorem{prop}[thm]{Proposition}
\newtheorem{cor}[thm]{Corollary}

\theoremstyle{remark}
\newtheorem{rmk}[thm]{Remark}
\newtheorem{example}[thm]{Example}

\theoremstyle{definition}
\newtheorem{defi}[thm]{Definition}

\numberwithin{thm}{section} 
\numberwithin{equation}{section}

\makeatletter

\newcommand{\Rmnum}[1]{\expandafter\@slowromancap\romannumeral #1@}
\makeatother

\def\R{{\mathbb R}}

\def\N{{\mathbb N}}
\def\H{{\mathbb H}}

\def\S{{\mathbf S}}

\newcommand{\vep}{\varepsilon}
\newcommand{\ol}{\overline}

\newcommand{\St}{\tilde{S}}
\newcommand{\Ht}{\tilde{\mathbb{H}}}

\newcommand{\tr}{\operatorname{tr}}

\newcommand{\spa}{\operatorname{span}}

\newcommand{\env}{\Gamma u}

\newcommand{\bpm}{\begin{pmatrix}}
\newcommand{\epm}{\end{pmatrix}}
\newcommand{\la}{\left\langle}
\newcommand{\ra}{\right\rangle}

plus 6pt minus 12pt
\title[Horizontal convex envelope]{\protect{Horizontal convex envelope in the Heisenberg group and applications to sub-elliptic equations}}

\author[Q. Liu]{Qing Liu}
\address{Qing Liu, Department of Applied Mathematics, Fukuoka University, Fukuoka 814-0180, Japan, {\tt qingliu@fukuoka-u.ac.jp}}
\author[X. Zhou]{Xiaodan Zhou}
\address{Xiaodan Zhou, Department of Mathematical Sciences, Worcester Polytechnic Institute, Worcester, MA 01609, USA, {\tt xzhou3@wpi.edu}}

\date{\today}

\begin{document}

\begin{abstract}
This paper introduces in a natural way a notion of horizontal convex envelopes of continuous functions in the Heisenberg group. We provide a convexification process to find the envelope in a constructive manner. We also apply the convexification process to show h-convexity of viscosity solutions to a class of fully nonlinear elliptic equations in the Heisenberg group satisfying a certain symmetry condition. Our examples show that in general one cannot expect h-convexity of solutions without the symmetry condition. 
\end{abstract}

\subjclass[2010]{35R03, 35D40, 26B25, 22E30}
\keywords{Heisenberg group, h-convex, viscosity solutions}

\maketitle

\section{Introduction}

The convex envelope of a given continuous function in $\R^N$ is a powerful tool in analysis and partial differential equations. In this paper, we exploit its sub-Riemannian counterpart, introducing the notion of convex envelope in the first Heisenberg group $\H$ and discussing its applications in the study of fully nonlinear sub-elliptic partial differential equations. 

\subsection{Background and motivation}

In order to explain the motivation of our work, let us first briefly recall the definition, properties and several applications of the convex envelope in $\R^N$.   For any  given  function $u\in C(\R^N)$ that is bounded below. There are at least  two ways to define the Euclidean convex envelope, which we denote by $\Gamma_E u$. The first is to consider the largest convex function majorized by $u$, that is, 
\begin{equation}\label{euclidean envelope1}
(\Gamma_E u)(p):= \sup\left\{v(p): \text{$v$ is convex and $v\leq u$ in $\R^N$}\right\} 
\end{equation}
for all $p\in \R^N$.

An equivalent way of defining the convex envelope is to convexify pointwise the given function $u$; namely, we have
\begin{equation}\label{euclidean envelope2}
\begin{aligned}
(\Gamma_E u)(p)=\inf\bigg\{\sum_{i=1}^{N+1}c_i u(p_i):\  c_i\in [0, 1], \  p_i\in \R^N\ (i=& 1, 2, \ldots, N+1),\\
& \sum_i c_i=1, \ \sum_i c_i p_i=p \bigg\}
\end{aligned}
\end{equation}
for all $p\in \R^N$. Compared to \eqref{euclidean envelope1}, the definition in \eqref{euclidean envelope2} is more constructive and more likely to be used in practical computations. 

Besides the equivalent definitions above, there is a characterization of the convex envelope  in terms of a nonlinear obstacle problem, recently proposed by \cite{Ob, ObS}. More precisely, in view of \cite[Theorem 2]{Ob}, the convex envelope $\Gamma_E$ can be characterized as a maximal viscosity solution of
\begin{equation}\label{euclidean ob}
\max \{-\lambda^\ast_E[v](x), v-u\}=0 \quad \text{ in $\R^N$,}
\end{equation}
that is, 
\[
(\Gamma_E u)(p)=\sup\left\{v(p): \text{$v$ is a subsolution of \eqref{euclidean ob}}\right\}.
\]
Here $\lambda^\ast_E[v]$ denotes the least eigenvalue of $\nabla^2 v$ for any $v\in C^2(\R^N)$. This can be viewed as an alternative expression of \eqref{euclidean envelope1}.

As an important tool, the convex envelope in $\R^N$ is extensively studied and widely applied in different contexts. One of important applications appears in the Alexandrov-Bakelman-Pucci (ABP) estimate for elliptic partial differential equations (see for instance \cite[Definition 3.1, Theorem 3.2]{CCbook}). The convex envelope is used to describe the coincidence set $\left\{p: \Gamma_E u(p)=u(p)\right\}$ contained in the domain, which leads to an accurate form of the estimate. 

Moreover, one can also use the convex envelope to show convexity of solutions to general fully nonlinear elliptic equations in the form
\begin{equation}\label{intro general eqn}
F(p, u(p), \nabla u(p), \nabla^2 u(p))=0\quad \text{in  $\R^N$,}
\end{equation}
where $F: \R^N\times \R\times \R^N\times \S^N\to \R$ is a continuous elliptic operator. Here $\S^N$ denotes the set of all $n\times n$ symmetric matrices. 

Two methods are well known to prove spatial convexity of the unique solution to an elliptic or parabolic equation. One method is based on the so-called convexity (concavity) maximum principle. 
For more details, we refer to \cite{Kor, K4, Ke} on this method for classical solutions and to \cite{GGIS} for a generalized result using viscosity solutions. 

The other method, proposed in \cite{ALL}, is to employ the minimization in \eqref{euclidean envelope2} to find the relation between the first and second derivatives of $\Gamma_E u$ at $p$ and those of $u$ at $p_i$. 
Here we assume that the infimum in \eqref{euclidean envelope2} can be attained at $p_i$ ($i=1, 2, \ldots, N+1$). Roughly speaking, assuming that these functions are smooth, we can easily see that 
\begin{equation}\label{intro all1}
\nabla \Gamma_E u(p)=\nabla u(p_i), \quad \text{for $i=1, 2, \ldots, N+1$,\  and}
\end{equation}
\[
\nabla^2 \Gamma_E u(p)\geq  \sum_i c_i \nabla^2 u(p_i). 
\]
Then under necessary regularity and  assumptions on the operator $F$,  one can use this relation to show that the convex envelope of any viscosity supersolution remains to be a viscosity supersolution \cite[Proposition 3]{ALL}. Such a supersolution preserving property enables us to obtain the convexity of the solution immediately if the comparison principle for the equation is known to hold. We refer the reader also to \cite{Ju} and recent work \cite{IshLS, CrFr3} for more applications of this method in the Euclidean space. 

In the sub-Riemannian setting, an intrinsic notion of convex functions is available. The so-called horizontal convex (h-convex) functions on the Heisenberg group is introduced in \cite{LMS} and on general Carnot groups in \cite{DGN1}. For a smooth function $u$ in $\H$, the h-convexity of $u$ simply requires $u$ to satisfy
\[
(\nabla_H^2 u)^\star\geq 0 \quad \text{in $\H$, }
\]
where $(\nabla_H^2 u)^\star$ stands for the symmetrized horizontal Hessian of $u$.  When $u$ is only a continuous function, then one should interpret the inequality above in the viscosity sense. Regularity properties of h-convex functions can also be found in  \cite{BaRi} for the Heisenberg group and in \cite{Wa, Mag, JLMS, MaSc, BaDr} for more general sub-Riemannian manifolds.

Such a convexity notion enables us to naturally consider the corresponding envelope for a given continuous function $u: \H\to \R$, following the definition \eqref{euclidean envelope1} in the Euclidean case. It is thus expected that the horizontal convex envelope can shed light on the above problems in the Heisenberg group. We remark that there are results on the ABP estimate in the sub-Riemannian circumstances such as \cite{DGN2, GuMo, GaTo, BaCaKr} but mainly for h-convex functions. 

We are thus more interested in the question: under what assumptions on the elliptic or parabolic equations are their solutions h-convex in space? As one of possible important applications, we aim to understand whether the h-convexity preserving property holds for the horizontal mean curvature flow in the Heisenberg group. Well-posedness for the level-set horizontal mean curvature flow equation is addressed in \cite{CC1, FLM1, BCi}. 
Under additional symmetry assumptions on solutions, the h-convexity preserving property for a class of semilinear parabolic equations is shown in \cite{LMZ} by extending a convexity maximum principle to the Heisenberg group. 

In this work, we focus our attention on fully nonlinear elliptic equations in the Heisenberg group in the form
\begin{equation}\label{elliptic eqn}
F(p, u, \nabla_H u, (\nabla_H^2 u)^\star)=0\quad \text{in  $\H$,}
\end{equation}
where $\nabla_H u$ denotes the horizontal gradient of $u$.  It is of our curiosity whether an analogue of the result in \cite{ALL} can bring us better convexity results for such general sub-elliptic equations. Our main purpose is therefore to study fundamental properties of the h-convex envelope and adopt them to understand geometric properties of 
\eqref{elliptic eqn}.  Let us mention that starshapedness of level sets of solutions to general elliptic equations in Carnot groups is recently studied in \cite{DGS}.

It is worth remarking that another interesting question concerns the continuous differentiability of the h-convex envelope. The regularity issue in the Euclidean space is addressed in \cite{GR, KiKr}. We discuss the same question in the Heisenberg group in our forthcoming work \cite{LZ3}.



\subsection{Main results}
As mentioned above, the most reasonable way to define the horizontal convex envelope (h-convex envelope), denoted by $\env$,  is clearly to take the greatest $h$-convex function majorized by $u$, i.e., 
\begin{equation}\label{envelope1}
(\env)(p):= \sup\left\{v(p): \text{$v$ is h-convex and $v\leq u$ in $\H$}\right\}
\end{equation}
for all $p\in \H$; see also Definition \ref{defi envelope1}.
The corresponding construction of $\env$ as in \eqref{euclidean envelope2} is not as straightforward as its definition. One may still attempt to convexify $u$ at each $p\in \H$ by setting 
\begin{equation}\label{envelope op}
S[u](p):=\inf\left\{\sum_{i}c_i u(p_i):\  c_i\in [0, 1], \ p_i\in \H_p\ (i=1, 2, 3),\ \sum_i c_i=1, \ \sum_i c_i p_i=p \right\},
\end{equation}
where $\H_p$ denotes the horizontal plane passing through $p$.
However, $S[u]$ is in fact not necessarily h-convex in $\H$ in general, as shown in Examples \ref{convexify ex2}, \ref{ex:two step} and \ref{ex:failure with coercivity}. Moreover, our examples also show that, without coercivity condition on $u$,  $S[u]$ may not be continuous in spite of the continuity of $u$. The operator $S$ only partially convexifies the function $u$ and the situation is thus totally different from the Euclidean case. 

It turns out that, in order to construct $\env$, one needs to iterate the operator $S$; namely, we show that 
\begin{equation}\label{iteration limit}
S^n[u]\to \env  \quad \text{ pointwise  in $\H$ as $n\to \infty$}
\end{equation}
provided that $u$ is continuous and bounded below in $\H$. The convergence is locally uniformly if $u$ is further assumed to be coercive in $\H$, i.e., 
\begin{equation}\label{coercive}
\inf_{|p|\geq R} {u(p)\over |p|}\to \infty\quad \text{ as $R\to \infty$.}
\end{equation} 
See Theorem \ref{thm equivalence} and Remark \ref{rmk uniforml convergence2} for precise statements. 

It causes much difficulty that the operator $S$ needs to be implemented multiple times to get the envelope $\env$. 
We do not know whether the iteration process can be completed in finite times. Example \ref{ex:one step} indicates that $\env$ can be obtained in one step for certain functions $u$ while Example \ref{ex:two step} shows that two steps are needed sometimes. It is not clear to us how the total number of the necessary iteration is related to the structure of the function $u$. 

We next discuss the application of the h-convex envelope in relation to the h-convexity of solutions to elliptic equations in the Heisenberg group. It was pointed out in our earlier work \cite{LMZ} that the horizontal convexity preserving property fails to hold even for linear transport equations. We now can give an example, showing that the h-convexity also fails for solutions of linear sub-elliptic equations in the form of 
\begin{equation}\label{linear sub-elliptic1}
u-\Delta_H u+\la\zeta, \nabla_H u \ra=f(p) \quad \text{in $\H$,}
\end{equation}
where $\zeta\in \R^2$ and $f\in C(\H)$ is given. Here $\Delta_H u$ denotes the horizontal Laplacian of $u$. Note that for the same type of equations in $\R^N$, i.e, 
\[
u-\Delta u+\la\zeta, \nabla u \ra=f(p) \quad \text{ in $\R^N$,}
\]
where in this case $\zeta\in \R^N$ is given, any smooth solution is convex provided that $f$ is convex in $\R^N$; the proof is merely an application of the maximum principle to the function $\la \nabla^2 u\ \eta, \eta\ra$ for any fixed $\eta\in \R^N$. 

However, similar convexity results cannot be expected for \eqref{linear sub-elliptic1}, since horizontal differentiation is not commutative in general. 
\begin{example}[Failure of h-convexity]\label{example no symmetry}
Consider the linear equation \eqref{linear sub-elliptic1} with $\zeta=(0, 2)$ and
\[
f(x, y, z)=2xz+x^2y+{1\over 4}x^4+{3\over 2}y^2+6y-1
\]
for $(x, y, z)\in \H$. By direct calculations, one can verify that 
\begin{equation}\label{example no symmetry sol}
u(x, y, z)=2xz+x^2y +{1\over 4}x^4-x^2+{3\over 2}y^2
\end{equation}
is the unique solution. (The uniqueness of solutions to this equation with polynomial growth at infinity is due to \cite[Theorem 7.4]{I5}.) Note that 
\[
(\nabla_H^2 f)^\star(x, y, z)=\begin{pmatrix}3x^2 & 3x \\ 3x & 3\end{pmatrix}
\]
but 
\[
(\nabla_H^2 u)^\star(x, y, z)=\begin{pmatrix}3x^2-2 & 3x \\ 3x & 3\end{pmatrix},
\]
which reveals that $u$ is not h-convex at the origin although $f$ is h-convex in $\H$. 
\end{example}
 This example suggests that  more restrictive assumptions on $F$ are needed if one wants to prove h-convexity of the solutions of \eqref{elliptic eqn}. A typical result we can show is as follows. 
 
 \begin{thm}[H-convexity for semilinear ellitpic equations]\label{intro thm}
Let $\alpha, \beta\geq 0$. Assume that $f\in C(\H)$ is h-convex in $\H$ and symmetric with respect to $z$-axis, i.e., 
\begin{equation}\label{zsym}		
u(x, y, z)=u(-x, -y, z), \quad \text{for all $(x, y, z)\in \H$.}		
\end{equation}
Let $\mathcal{A}\subset \R^2$ be a compact set symmetric with respect to the origin, that is, $\zeta\in \mathcal{A}$ implies $-\zeta\in \mathcal{A}$. If $u$ is a coercive solution of the semilinear equation
 \begin{equation}\label{special eqn}
u=\alpha\Delta_H u+\beta\sup_{\zeta\in \mathcal{A}}\la \zeta, \nabla_H u\ra+f(p) \quad \text{in $\H$,}
\end{equation}
then $\env$ is a supersolution. In particular, the unique solution of \eqref{special eqn} with polynomial growth near space infinity is h-convex in $\H$. 
 \end{thm}
The polynomial growth condition here is again used to guarantee the uniqueness of solutions due to the comparison result in \cite[Theorem 7.4]{I5}; see Section \ref{sec:viscosity} for clarification. 
 
 Note that Theorem \ref{intro thm} is only a special case of our main result, where we prove the h-convexity of the solution to \eqref{elliptic eqn} with a general nonlinear concave symmetric operator $F$, assuming that the comparison principle holds; see Theorem \ref{cor h-convex}.  
We emphasize that  $F$ here is assumed to be concave with respect to all arguments, which is stronger than the condition in the Euclidean case. The reasons why the symmetry and the strong concavity assumptions are needed will be clarified in a moment.

Compared to the left invariant h-convexity, it is in fact easier to obtain the desired results by considering the right invariant case instead. To see this, we introduce the right invariant counterpart $\tilde{S}$ of the partial convexification operator $S$, given by
\begin{equation}\label{envelope op right}
\St[u](p):=\inf\left\{\sum_{i}c_i u(p_i):\  c_i\in [0, 1], \ p_i\in \Ht_p\ (i=1, 2, 3),\ \sum_i c_i=1, \ \sum_i c_i p_i=p \right\},
\end{equation}
where $\Ht_p$ stands for the right invariant horizontal plane passing though $p\in \H$.  We then show, in Theorem \ref{thm r-preserve},  that $\tilde{S}$ has the supersolution preserving property for a class of concave fully nonlinear elliptic operators, i.e., $\tilde{S}[u]$ is a supersolution if $u$ is a supersolution. 

Our proof of Theorem \ref{thm r-preserve} is essentially an adaptation of the argument \cite{ALL} to the sub-Riemannian setting. However, in contrast to the situation in \cite{ALL}, we here have an additional constraint condition due to the extra requirement $p_i\in \tilde{\H}_p$ in \eqref{envelope op right}. Roughly speaking, for any fixed $p\in \H$ and minimizers $p_i\in \H$ in \eqref{envelope op right}, in order to find the relation between $\left((\nabla_H \Gamma u)(p), (\nabla^2_H \Gamma u)^\star(p)\right)$ and $\left(\nabla_H u(p_i), (\nabla^2_H u)^\star(p_i)\right)$, we append to the standard minimization an extra term penalizing the distance between $p_i$ and the horizontal plane through $\sum_i p_i$. We refer the reader to Section \ref{sec:right convex} for technical details.  In summary, our idea more closely resembles the method of Lagrange multipliers rather than direct unconstrained minimization. 

It is the extra penalty term that requires us to impose the concavity of $\xi \mapsto F(p, r, \xi, A)$, which is not needed in the Euclidean case. Roughly speaking, we are not able to obtain the equality as in \eqref{intro all1} for the horizontal gradients but instead we get 
\[
\nabla_H \Gamma u(p)=\sum_i c_i \nabla_H u(p_i),
\]
which demands the concavity of $F$ in the horizontal gradient to conclude. In addition, we have an example, Example \ref{ex strong concavity}, showing that such a strong concavity condition is necessary. This condition unfortunately excludes possible applications of our approach to the mean curvature operator and p-Laplacians in the Heisenberg group. 

Let us briefly discuss the symmetry assumption on $F$. Notice that due to the presence of Example \ref{example no symmetry}, even in the simpler case \eqref{special eqn}, it seems necessary to assume the function $f$ and the term involving $\nabla_H u$ to be symmetric. 
The additional symmetry on the elliptic operator implies that the unique solution $u$ is symmetric about the $z$-axis. This further yields $S[u]=\tilde{S}[u]$ in $\H$; see Theorem \ref{thm axisym2} for details. We thus have $S[u]=u$, which concludes the proof of Theorem \ref{intro thm}. By iterating this argument for $S^n[u]$ and passing to the limit as $n\to \infty$, we can also prove a symmetric supersolution preserving property, namely, if $u$ is a symmetric supersolution, then so is  $\env$. This property is addressed for the general equation \eqref{elliptic eqn} in Theorem \ref{thm preserve rot}. 


We finally mention that one can also obtain the Euclidean convexity of the unique solution to \eqref{elliptic eqn} under weaker structure assumptions on $F$. In particular, the symmetry condition is no longer needed in this case. Our result is consistent with that in \cite{ALL} and our proof is based on slight modification of the arguments used for Theorem \ref{thm r-preserve} but applied to the Euclidean convex envelope. See Section \ref{sec:euclidean} for more explanations. 

The rest of the paper is organized in the following way. In Section \ref{sec:prep}, we give a brief review of the Heisenberg group and the theory of viscosity solutions.  We also recall basic notions and regularity results related to the horizontal convex functions in Section \ref{sec:convex}. The definition of h-convex envelop is given in Section \ref{sec:defi}. We introduce the iterated convexification process and give several concrete examples of the envelope in Section \ref{sec:convexify}. Section \ref{sec:elliptic} is devoted to our main results with detailed discussion on the application of the h-convex envelope to the study on sub-elliptic PDEs.

\subsection*{Acknowledgments}

The work of the first author was partially supported by Grant-in-Aid for Scientific Research (C) (No. 19K03574) from Japan Society for the Promotion of Science and by the Grant from Central Research Institute of Fukuoka University (No. 177102). The work of the second author was partially supported by an AMS Simons Travel Grant.

\section{Preparations}\label{sec:prep}

\subsection{Preliminaries on the Heisenberg group}\label{sec:Heisenberg}

Recall that the Heisenberg group $\mathbb{H}$ is $\mathbb{R}^{3}$ endowed with the non-commutative group multiplication 
\[
(x_p, y_p, z_p)\cdot (x_q, y_q, z_q)=\left(x_p+x_q, y_p+y_q, z_p+z_q+\frac{1}{2}(x_py_q-x_qy_p)\right),
\]
for all $p=(x_p, y_p, z_p)$ and $q=(x_q, y_q, z_q)$ in $\H$. Note that the group inverse
of $p=(x_q, y_q, z_q)$ is  $p^{-1}=(-x_q, -y_q, -z_q)$. The Kor\'anyi gauge is given by
\[
|p|_G=((p_1^2+p_2^2)^2+16 p_3^2)^{1/4},
\]
and the left-invariant Kor\'anyi or gauge metric is 
\[
d_L(p, q)=|p^{-1}\cdot q|_G.
\]
We denote by $B_R(p)$ the gauge ball centered at $p$ with radius $R>0$. We denote the gauge ball centered at the origin simply by $B_R$. 

The Lie algebra of $\mathbb{H}$ is generated by the left-invariant vector fields
\[
X={\partial\over \partial x}-\frac{y}{2}{\partial\over \partial z}, \quad  Y={\partial\over \partial y}+\frac{x}{2}{\partial\over \partial z}, \quad Z={\partial \over \partial z}.
\]
  One may easily verify the commuting relation $Z=[X, Y]=
  XY- YX$.
  
The horizontal gradient  of $u$ is given by 
\[
\nabla_H u=(X u, Y u)
\]
and the symmetrized second horizontal Hessian  $(\nabla_H^2 u)^\ast\in \S^{2}$ is given by  
\[
(\nabla_H^2 u)^\star:=\left(\begin{array}{cc} X^2 u & (XY u+YX u)/2\\
(XY u+YX u)/2 & Y^2 u\end{array}\right).
\]
Here $\S^{n}$ denotes the set of all $n\times n$ symmetric matrices.

Define
\[
\mathbb{H}_0=\{h\in \mathbb{H}: h=(x, y, 0) \ \text{ for $x, y\in \mathbb{R}$}\}.
\]
For any $p\in \H$, we call
\[
\H_p=\{p\cdot h: \ h\in \H_0\}
\]
the horizontal plane through $p$. 
The horizontal plane through $p=(x_p, y_p, z_p)$ can be expressed by the following equation:
\begin{equation}\label{eq plane}
y_p x-x_py+2z-2z_p=0.
\end{equation}

\begin{rmk}\label{rmk right vector}
Later we will also use the right invariant vector fields in $\H$ given by 
\[
\tilde{X}={\partial \over \partial x}+\frac{y}{2}{\partial \over \partial z}, \quad  \tilde{Y}={\partial\over \partial y}-\frac{x}{2}{\partial\over \partial z}, \quad \tilde{Z}={\partial \over \partial z}.
\]
Accordingly,  the right invariant horizontal gradient 
\[
\tilde{\nabla}_H u=(\tilde{X} u, \tilde{Y} u)
\]
and symmetrized Hessian
\[
(\tilde{\nabla}_H^2 u)^\star:=\left(\begin{array}{cc} \tilde{X}^2 u & (\tilde{X}\tilde{Y} u+\tilde{Y}\tilde{X} u)/2\\
(\tilde{X}\tilde{Y} u+\tilde{Y}\tilde{X} u)/2 & \tilde{Y}^2 u\end{array}\right).
\]
We also write $\tilde{\H}_p$ to denote the right invariant horizontal plane, that is, 
\begin{equation}\label{right plane}
\tilde{\H}_p=\{h\cdot p: h\in \H_0\}. 
\end{equation}
We can write an analogue of the plane equation \eqref{eq plane} for $\tilde{\H}_p$ as below:
\begin{equation}\label{eq plane right}
y_p x-x_py-2z+2z_p=0.
\end{equation}
\end{rmk}

\subsection{Viscosity solutions}\label{sec:viscosity}
Viscosity solutions are known to have many applications in fully nonlinear equations in the Euclidean space; see \cite{CIL} for an introduction.  We refer to \cite{Bi1, Mnotes} and many others for generalization on the sub-Riemannian manifolds. 

Let us consider \eqref{elliptic eqn},
where $F: \H\times \R \times \R^2\times \S^2\to \R$ is a continuous operator satisfying the following assumptions. 
\begin{enumerate}
\item[(A1)] $F$ is  (degenerate) elliptic; namely, 
\[
F(p, r, \xi, A_1)\leq F(p, r, \xi, A_2)
\]
for all $p\in \H$, $r\in \R$, $\xi\in \R^2$ and $A_1, A_2\in \S^2$ with $A_1\geq A_2$. 

\item[(A2)] $F$ is proper; namely, 
\[
F(p, r_1, \xi, A)\geq F(p, r_2, \xi, A)
\]
for all $p\in \H, \xi\in \R^2, A\in \S^2$ and $r_1, r_2\in \R$ with $r_1\geq r_2$.

\end{enumerate}

We begin with a definition for viscosity solutions of \eqref{elliptic eqn} below. Denote by $USC(\H)$ (resp., $LSC(\H)$) the class of upper semicontinuous  (resp., lower semicontinuous) functions in $\H$.

\begin{defi}[Definition of viscosity solutions]
Let $\Omega$ be a domain in $\H$. A locally bounded function $u\in USC(\Omega)$ (resp., $u\in LSC(\Omega)$) is said to be a viscosity subsolution (resp., supersolution) of \eqref{elliptic eqn} in $\Omega$ if whenever there exist $\varphi\in C^2(\Omega)$ and $p_0\in \Omega$ such that $u-\varphi$ attains a (strict) maximum (resp., minimum) in $\Omega$ at $p_0$, we have 
\[
\begin{aligned}
&F\left(p_0, u(p_0), \nabla_H \varphi(p_0), (\nabla_H^2 \varphi)^\star(p_0)\right)\leq 0\\
&\quad \left(\text{resp., } F\left(p_0, u(p_0), \nabla_H \varphi(x_0), (\nabla_H^2 \varphi)^\star(p_0)\right)\geq 0 \right).
\end{aligned}
\]
A bounded continuous function $u$ is called a viscosity solution of \eqref{elliptic eqn} if it is both a subsolution and  a supersolution.
\end{defi}

As a standard remark, one can use the so-called subelliptic semijets to give an alternative definition of sub- and supersolutions; we refer the reader to \cite{Bi1, Mnotes} for details. Hereafter let us denote by $J^{2, \pm}_H u(p)$ the semijets of $u$ at a given point $p\in \H$. Moreover, we use $\overline{J}^{2, \pm}_H u(p)$ to denote the ``closure'' of $J^{2, \pm}_H u(p)$. We can equivalently define a supersolution by requiring that, for any $p_0\in \H$,
\[
F(p_0, u(p_0), \xi, A)\geq 0
\]
holds provided that $(\xi, A)\in \overline{J}^{2, -}_H u(p_0)$; see \cite[Proposition 3.1]{Bi1}. One can give an equivalent definition for subsolutions analogously. 

\begin{rmk}\label{rmk comparison}
Throughout this work, we always assume that a comparison principle holds for \eqref{elliptic eqn}, since it is not our main concern here. Let us recall the standard comparison principle states that any subsolution $u$ and any supersolution $v$ of \eqref{elliptic eqn} satisfies $u\leq v$ in $\H$. However, it is worth remarking that establishing a comparison principle for fully nonlinear elliptic equations in the whole space is completely nontrivial even in the Euclidean space. One usually needs to impose additional assumptions on the growth rate of sub- and supersolutions near space infinity. 

However, on the other hand, a comparison principle is available in \cite{I5} for a special class of operators 
\begin{equation}\label{special form}
F(p, r, \xi, A)=r-\alpha\tr A-\beta\sup_{\zeta\in \mathcal{A}}\la \zeta, \xi\ra-f(p)
\end{equation}
for $(p, r, \xi, A)\in \H\times \R\times \R^2\times \S^2$, where $\alpha, \beta\geq 0$, $\mathcal{A}$ is a compact subset of $\R^2$ and $f\in C(\H)$  is given. 
Indeed, in this case one can write \eqref{elliptic eqn} in the Euclidean coordinates and apply \cite[Theorem 7.4]{I5} to get $u\leq v$ in $\H$ if the subsolution $u$ and the supersolution $v$ have polynomial growth at space infinity, namely, there exists $k>0$ such that 
\[
\sup_{p\in \H}{u(p)\over |p|^k+1}<\infty \quad \text{and}\quad \inf_{p\in \H}{v(p)\over |p|^k+1}>-\infty. 
\]
\end{rmk}

Viscosity solutions to the parabolic equation,
\begin{equation}\label{parabolic eqn}
u_t+F(p, u, \nabla_H u, (\nabla_H^2 u)^\star)=0 \quad \text{ in $\H\times (0, \infty)$,}
\end{equation}
can be similarly defined. 

\subsection{Horizontal convexity}\label{sec:convex}
In what follows, we review basic results on the notion of convex functions in the Heisenberg group; more details can be found in \cite{LMS, DGN1}. 

\begin{defi}[{\cite[Definition 4.1]{LMS}}]\label{defi h-convex}
Let $\Omega$ be an open set in $\mathbb{H}$ and $u: \Omega\to \mathbb{R}$ be an upper semicontinuous function. The function $u$ is said to be horizontally convex or h-convex in $\Omega$, if for every $p\in \H$ and $h\in \H_0$ such that $[p\cdot h^{-1}, p\cdot h]\subset \Omega$, we have 
\begin{equation}\label{h-convex eqn}
u(p\cdot h^{-1})+u(p\cdot h)\geq 2u(p).
\end{equation}
\end{defi}

One may also define convexity of a function in the following way. 
\begin{defi}[{\cite[Definition 3.3]{LMS}}]\label{defi v-convex}
Let $\Omega$ be an open set in $\mathbb{H}$ and $u:\Omega\to \mathbb{R}$ be an upper semicontinuous function. The function $u$ is said to be v-convex in $\Omega$ if
\begin{equation}\label{v-convex eqn}
(\nabla^2_H u)^\star(p)\geq 0  \quad \text{for all $p\in \Omega$}
\end{equation}
in the viscosity sense. 
\end{defi}
It is easily seen that $u\in C^2(\Omega)$ is v-convex if it satisfies \eqref{v-convex eqn} everywhere in $\Omega$. It is known that the h-convexity and v-convexity are equivalent \cite{LMS}; see also the related results in Carnot groups \cite{Wa, JLMS}. Below we also review a well-known result concerning the Lipschitz regularity of h-convex functions. 
\begin{thm}[Local Lipschitz regularity of h-convex functions {\cite[Theorem 3.1]{LMS}}]\label{thm lip}
Suppose that $u: \H\to \R$ is an h-convex (v-convex) function. Then $u$ is locally bounded and Lipschitz. Moreover, the following estimate holds:
\[
\|\nabla_H u\|_{L^\infty(B_R)}\leq {C\over R}\|u\|_{L^{\infty}(B_{2R})}.
\]
Here $C>0$ is independent of $u$ and $R$. 
\end{thm}
We remark that the original result in \cite[Theorem 3.1]{LMS} is stated for a general domain $\Omega\subset \H$. Here we only consider the special case $\Omega=\H$ for our own purpose. We refer to \cite{LMS, DGN1, BaRi} for details on this result and to \cite{Wa, Mag, MaSc, BaDr} for further discussion on general sub-Riemannian manifolds.

\begin{rmk}\label{rmk right convex}
We can also consider a right invariant version of h-convex functions in $\H$ by using the vector field in Remark \ref{rmk right vector}. More precisely, we say a function is right invariant h-convex in an open set  $\Omega\subset \H$ if for every $p\in \H$ and $h\in \H_0$ such that $[h^{-1}\cdot p, h\cdot p]\subset\Omega$, we get
\[
u(h^{-1}\cdot p)+u(h\cdot p)\geq 2u(p).
\]
Equivalently, we may also use the viscosity inequality 
\[
(\tilde{\nabla}^2_H u)^\ast(p)\geq 0  \quad \text{for all $p\in \Omega$.}
\]
Applying an argument symmetric to the proof of Theorem \ref{thm lip}, we can show that the right invariant h-convex function is also locally bounded and Lipschitz (with respect to the right invariant metric). 
\end{rmk}

In general, h-convex functions or right invariant h-convex functions are not necessarily convex in the Euclidean sense, as shown in the following example. 
\begin{example}
Let $u(p)=x^2 y^2+2z^2$. It is easily verified that $u$ is an h-convex and right invariant h-convex function that is not Euclidean convex. 
\end{example}
The above example also clearly indicates that a function that is both left invariant h-convex and right invariant h-convex may not be Euclidean convex.

\section{Definition of h-convex envelope}\label{sec:defi}

In this section we aim to extend the definitions of Euclidean convex envelopes to the Heisenberg group.

To define a horizontal convex envelope of $u\in C(\H)$ bounded below, we may follow Perron's method and consider a sub-Riemannian analogue of \eqref{euclidean envelope1} as follows. 

\begin{defi}\label{defi envelope1}
Suppose that $u\in C(\H)$ is bounded below. A function $\env: \H\to \R$ is said to be the h-convex envelope of $u$ if $\env$ is the greatest $h$-convex function majorized by $u$, i.e., $\env$ is given by \eqref{envelope1}.
\end{defi}

The function $\env$ is well-defined, since $u$ is bounded below and any constant is h-convex.  It is clear that 
\[
\inf_{\H} u\leq \env\leq u \quad \text{in $\H$.}
\]

It is not difficult to show that $\env$ is locally Lipschitz in $\H$ due to the Lipschitz estimate of h-convex functions in Theorem \ref{thm lip}. One may also obtain the local Lipchitz continuity of $\env$ by first showing its h-convexity, as below, and then applying Theorem \ref{thm lip}.

\begin{lem}[H-convexity of the envelope]
Suppose that $u\in C(\H)$ is bounded from below. Let $\env$ be given as in \eqref{envelope1}. Then $\env$ is h-convex in $\H$.  
\end{lem}
\begin{proof}
As mentioned above, $\env$ is locally Lipschitz in $\H$. The proof of the h-convexity of $\env$ streamlines the argument of Perron's method. More precisely, if there exists $\varphi\in C^2(\H)$ and $p_0$ such that $\env-\varphi$ attains a strict maximum at $p_0$, then we may find $v_j\in C(\H)$ h-convex and $p_j\in \H$ with $p_j\to p_0$ as $j\to \infty$ such that $v_j-\varphi$ attains a local maximum at $p_j$. By the h-convexity of $v_j$, we obtain that 
\[
(\nabla^2_H \varphi)^\star(p_j) \geq 0,
\]
from which we deduce that 
\[
(\nabla^2_H \varphi)^\star(p_0) \geq 0
\]
by passing to the limit as $j\to \infty$.
\end{proof}

Motivated by \cite{Ob}, we may consider the following obstacle problem
\begin{equation}\label{obs prob}
\max\{-\lambda^\ast[v],\ v-u\}=0 \quad \text{ in $\H$}
\end{equation}
in the viscosity sense, where $\lambda^\ast[v]$ denotes the least eigenvalue of $(\nabla^2_H v)^\star$. 
\begin{thm}[Characterization by an obstacle problem]
Assume that $u\in C(\H)$ is bounded from below. Let $\env$ be the h-convex envelope defined in \eqref{envelope1}. Then 
\[
(\env)(p)=\sup\{v(p): \text{$v$ is a subsolution of \eqref{obs prob}}\}.
\]
\end{thm}
We omit the proof, since it is merely a direct adaptation of \cite[Theorem 2]{Ob} to the sub-Riemannian circumstances based on \eqref{envelope1}.

We shall give several concrete examples in Section \ref{sec:example} and discuss an application to convexity of solutions to nonlinear PDEs in Section \ref{sec:elliptic}. 

\section{Pointwise Convexification}\label{sec:convexify} 

\subsection{A partially convexifying operator}\label{sec:convexify1}

We next use the convexification similar to \eqref{euclidean envelope2} to find the horizontal convex envelope. In $\H$, it is natural to consider an operator $S$ as given by \eqref{envelope op} for any $u\in USC(\H)$ bounded from below.
In contrast to the Euclidean case \eqref{euclidean envelope2}, the main difference here is that $p_i$ are restricted on the horizontal plane $\H_p$ rather than the whole space. 
It is obvious that
\[
\inf_{\H} u\leq S[u]\leq u \quad \text{in $\H$.}
\]
It is also clear that $S[u]=u$ in $\H$ if and only if $u$ is h-convex. As is explained later, $S[u]$ is not necessarily h-convex for an arbitrary $u\in C(\H)$; see Example \ref{convexify ex2}.

Let us now verify that $S$ maps $u\in USC(\H)$ to $S[u]\in USC(\H)$.
\begin{lem}[Upper semicontinuity preserving]\label{lem usc}
Suppose that $u\in USC(\H)$ is bounded from below. Let the operator $S$ be defined as in \eqref{envelope op}. Then $S[u]\in USC(\H)$.
\end{lem}
\begin{proof}
Fix $\ol{p}\in \H$ arbitrarily. In view of \eqref{envelope op}, for any $\vep>0$, there exist $\ol{c}_i\in [0, 1]$ and $\ol{p}_i\in \H_{\ol{p}} (i=1, 2, 3)$ such that 
\[
\sum_{i=1, 2, 3} \ol{c}_i=1, \quad \sum_{i=1, 2, 3} \ol{c}_i \ol{p}_i=\ol{p}
\]
and
\begin{equation}\label{usc eq1}
S[u](\ol{p})\geq \sum_{i=1, 2, 3}\ol{c}_i u(\ol{p}_i)-\vep.
\end{equation}
Moreover, by the continuity of $p\mapsto \H_p$ and the upper semicontinuity of $u$, for any $p\in \H$ sufficiently close to $\ol{p}$, we can find $p_i\in \H_p$ near $\ol{p}_i$ such that 
\[
\sum_{i=1, 2, 3}\ol{c}_i p_i=p
\]
and for all $i=1, 2, 3$
\[
u(p_i)\leq u(\ol{p}_i)+\vep.
\]
It follows that 
\[
\sum_i \ol{c}_i u(p_i)\leq \sum_i \ol{c}_i u(\ol{p}_i)+3\vep.
\]
By \eqref{usc eq1} and \eqref{envelope op}, we obtain that 
\[
S[u](p)\leq S[u](\ol{p})+4\vep,
\]
which implies the upper semicontinuity of $S[u]$. 
\end{proof}

However, $S$ does not preserve lower semicontinuity in general, which is very different from the Euclidean case \cite[Lemma 4]{ALL}. An example is as follows. 

\begin{example}[Loss of lower semicontinuity preserving]\label{convexify ex1}
Let us construct $u\in C(\H)$ bounded below satisfying 
\begin{equation}\label{lsc eq1}
u(p)=1 \quad \text{ for all $p\in \H_0$,  and}
\end{equation}
\begin{equation}\label{lsc eq2}
\lim_{x\to \infty} u(q_x)=\lim_{x\to \infty} u\left((q_x)^{-1}\right)=0, \quad \text{where $q_x=(x, 0, -1/2)$}.
\end{equation}
Take $\vep>0$ arbitrarily small and consider the point $p_\vep=(\vep, \vep, 0)\in \H$. It is not difficult to verify that
\[
p_+, p_-\in \spa\{X(p_\vep), Y(p_\vep)\}= \H_{p_\vep},
\] 
where
\[
p_+=\left({1\over\vep}, 0, -{1\over 2}\right), \quad p_-=\left(-{1\over \vep}, 0, {1\over 2}\right).
\] 
By the definition of $S$ in \eqref{envelope op}, we easily see that 
\[
S[u](p_\vep)\leq {1\over 2}u(p_+)+{1\over 2}u(p_-)
\]
and therefore by \eqref{lsc eq2} 
\[
\liminf_{\vep\to 0}S[u](p_\vep)\leq 0.
\]
On the other hand, by \eqref{lsc eq1} we deduce that $S[u](0)= 1$. Hence, we conclude that the lower semicontinuity of $S[u]$ fails to hold at the origin. 
\end{example}

The counterexample above would not exist if we could exclude the situation like  \eqref{lsc eq2}.  The lower semicontinuity for $S[u]$ can be obtained by assuming that $u$ is coercive, as indicated in the following result.

\begin{lem}[Lower semicontinuity preserving under coercivity]\label{lem lsc}
Suppose that $u\in LSC(\H)$ satisfies the coercivity condition \eqref{coercive}. 
Let the operator $S$ be defined as in \eqref{envelope op}. Then $S[u]\in LSC(\H)$.
\end{lem}
\begin{proof}
Without loss of generality, we may only show that 
\[
\liminf_{p\to 0} S[u](p)\geq S[u](0).
\]
For any sequence $\{p_j\}\subset \H$ with $p_j\to 0$ as $j\to \infty$, there exist $p_{ij}\in \H_{p_j}$ and $c_{ij}\in [0, 1]$ $(i=1, 2, 3)$ for each $j$ such that 
\begin{equation}\label{lsc1}
\sum_{i=1, 2, 3} c_{ij}=1, \quad \sum_{i=1, 2, 3}c_{ij} p_{ij}=p_j, \quad\text{and}\quad
\sum_{i=1, 2, 3} c_{ij} u(p_{ij})\leq S[u](p_j)+{1\over j}.
\end{equation}
Thanks to the coercivity condition \eqref{coercive}, we see that $\{p_{ij}\}_{i=1, 2, 3}$ are bounded uniformly in $j$. We thus can take a subsequence (still indexed by $j$ for simplicity) such that as $j\to \infty$
\[
c_{ij}\to \ol{c}_i\in [0, 1], \quad p_{ij}\to \ol{p}_i\in \H.
\]
It follows from \eqref{lsc1} that
\[
\sum_{i} \ol{c}_i=1, \quad \sum_{i} \ol{c}_i \ol{p}_i=0,
\] 
and, by lower semicontinuity of $u$, 
\[
\sum_{i=1, 2, 3} \ol{c}_{i} u(\ol{p}_{i})\leq \liminf_{j\to \infty} S[u](p_j).
\]

Moreover, by the locally uniform continuity of the horizontal plane $\H_p$ with respect to $p$, we have 
$\ol{p}_i\in \H_0$. By definition of $S$, we thus have
\[
S[u](0)\leq \liminf_{j\to \infty} S[u](p_j).
\]
This concludes the proof, since the converging sequence $\{p_j\}$ is taken arbitrarily. 
\end{proof}

 The following lemma is thus an immediate consequence. 
\begin{lem}[Preservation of coercivity]\label{lem coercive}
Suppose that $u\in C(\H)$ satisfies \eqref{coercive}. Let $S$ be given by \eqref{envelope op}. Then $S[u]$ also satisfies \eqref{coercive}.
\end{lem}
\begin{proof}
Note that the condition \eqref{coercive} implies that for each $C_1>0$, there exist $R>0$ and $C_2>0$ depending on $R$ such that
\[
u(p)\geq C_1|p|-C_2
\]
for all $p\in \H$.  It is clear that $S[u](p)\geq C_1|p|-C_2$ for all $p\in\H$, since the right hand side is convex in the Euclidean sense and therefore h-convex in $H$. It follows that
\[
\liminf_{R'\to \infty}\inf_{|p|\geq R'}{S[u](p)\over |p|}\geq C_1,
\]
which yields the coercivity of $S[u]$ thanks to the arbitrariness of $C_1>0$. 
\end{proof}

We finally remark that it is possible to define a right-invariant version of the operator $S$ as below. For any $u\in C(\H)$ bounded below, let $\St$ be  given by \eqref{envelope op right}. It is easily seen that the properties of $S$ above hold also for $\St$ via analogous arguments above.
In particular, $\St[u]$ is continuous and coercive in the sense of \eqref{coercive} in $\H$ provided that $u$ is continuous and coercive in the same sense.

\subsection{Iterated convexification}\label{sec:convexify2}
One may expect that $S[u]$ is h-convex in $\H$ for any $u\in C(\H)$, but it turns out to be false in general. An example similar to Example \ref{convexify ex1} can be easily built as below. 
\begin{example}[Failure of h-convexification]\label{convexify ex2}
Consider $p=(0, 0, 0)$ and $h=(1, 0, 0)\in \H_0$. Let 
\[
p_1=(1, 1, 1/ 2), \quad p_2=(1, -1, -1/2),
\]
\[
p_3=(-1, -1, 1/2), \quad p_4=(-1, 1, -1/2).
\]
Note that $p_i\notin \H_0$ for all $i=1, ..., 4$. We thus can take $u\in C(\H)$ such that $u(p_1)=u(p_2)=u(p_3)=u(p_4)=0$ and $u(p)=1$ for all $p\in \H_0$.  

Since $p_1, p_2\in \H_{h}$ and $p_3, p_4\in \H_{h^{-1}}$, by definition \eqref{envelope op}, we have 
\[
S[u](h)\leq {1\over 2} u(p_1)+{1\over 2} u(p_2), \quad S[u](h^{-1})\leq {1\over 2} u(p_3)+{1\over 2} u(p_4).
\]
On the other hand, since $u\equiv 1$ in $\H_0$, we get
\[
S[u](0)\geq 1. 
\]
It follows that 
\[
S[u](h)+S[u](h^{-1})<2S[u](0),
\]
which states that $S[u]$ is not h-convex at the origin. 

The function $u$ in this example is not coercive. However, it is not difficult to raise the values of $u$ near space infinity without influencing much the effect of the operator $S$ on $u$ at those points we are interested in. A more explicit example will given in Example \ref{ex:failure with coercivity} of Section \ref{sec:example}.
\end{example}

Although the operator $S$ does not give us the convex envelope immediately, we may iterate it and  passing to the limit.  In other words, we take
\begin{equation}\label{envelope2}
U(p):=\lim_{n\to \infty} S^n[u](p)
\end{equation}
for $u\in C(\H)$ and any $p\in \H$.  It is easily seen that the pointwise limit of $S^n[u]$ actually exists, thanks to the monotonicity $S[u]\leq u$ for $u\in USC(\H)$. In addition, we have 
\begin{equation}\label{h-convexify eq0}
U\leq u \quad \text{ in $\H$.}
\end{equation}

\begin{lem}[H-convexity of $U$]\label{lem h-convexify}
Suppose that $u\in C(\H)$ is bounded from below. Let $U$ be given as in \eqref{envelope2}. Then $U$ is h-convex in $\H$.  In particular, $U$ is locally Lipschitz with respect to the gauge metric in $\H$.
\end{lem}
\begin{proof}
Fix $p\in \H$ and $h\in \H_0$. By definition, for any $\vep>0$, there exists $n>0$ sufficiently large such that 
\begin{equation}\label{h-convexify eq1}
U(p\cdot h)\geq S^n[u](p\cdot h)-\vep, \quad  U(p\cdot h^{-1})\geq S^n[u](p\cdot h^{-1})-\vep.
\end{equation}
In view of \eqref{envelope op}, we have   
\[
S^n[u](p\cdot h)+S^n[u](p\cdot h^{-1})\geq 2S^{n+1}[u](p),
\]
which implies by monotonicity of $S$ that 
\[
S^n[u](p\cdot h)+S^n[u](p\cdot h^{-1})\geq 2U(p). 
\]
Combining this relation with \eqref{h-convexify eq1}, we deduce that 
\[
U(p\cdot h)+U(p\cdot h^{-1})\geq 2U(p)-2\vep.
\]
We conclude that $U$ is h-convex by letting $\vep\to 0$.

The Lipschitz regularity is an immediate consequence of Theorem \ref{thm lip}.
\end{proof}

\begin{rmk}\label{rmk uniform convergence}
Assuming in addition that $u\in C(\H)$ is coercive in the sense of \eqref{coercive}, we have $S^n[u]$ is continuous in $\H$. Since $S^n[u]$ is a monotone sequence of continuous functions and $U$ is continuous in $\H$, by Dini's theorem, we therefore obtain locally uniform convergence of $S^n[u]$ to $U$ in $\H$ as $n\to \infty$. 
\end{rmk}

We finally present our main theorem of this section, showing that $U$ and $\env$ are the same for any given $u\in C(\H)$.

\begin{thm}[Characterization by iterated convexification]\label{thm equivalence}
Suppose that $u\in C(\H)$ is bounded from below. Let $\env$ and $U$ be given by \eqref{envelope1} and \eqref{envelope2}. Then $\env\equiv U$ in $\H$.  
\end{thm}
\begin{proof}
We now prove the equivalence $\env=U$. Since $\env\leq u$, by the h-convexity of $\env$ we thus have for any $p\in \H$
\[
\env(p)\leq \sum_i c_i \env(p_i)\leq  \sum_i c_i u(p_i)
\]
for all $0\leq c_i\leq 1$, $p_i\in \H_p$ (i=1, 2, 3) with $\sum_i c_i=1$, $p_i\in \H$ and $\sum_i c_ip_i =p$. It follows that $\env\leq S[u]$ in $\H$. 

Since $S[\env]=\env$, iterating the argument above yields that 
\[
\env\leq S^n[u],
\]
which implies that $\env\leq U$.

Since $\env$ is defined to be the largest h-convex function below $u$, the inequality $\env\geq U$ is an immediate consequence of Lemma \ref{lem h-convexify} together with \eqref{h-convexify eq0}.
\end{proof}
\begin{rmk}\label{rmk uniforml convergence2}
Since $\env$ and $U$ are equivalent, from now on, {we also use $\env$ to denote the limit of $S^n[u]$ for any given function $u\in C(\H)$ bounded below}. In particular, in the presence of the coercivity \eqref{coercive} of $u$,  we have $S^n[u]\to U$ locally uniformly as $n\to \infty$ due to Remark \ref{rmk uniform convergence}. 
\end{rmk}

\subsection{Symmetry with respect to $z$-axis}\label{sec:axisymmetry}
We consider a special case when $u$ satisfies a symmetry condition. We say $u$ is symmetric with respect to the $z$-axis if \eqref{zsym} holds. 
We can show that the operator $S$ preserves this symmetry condition. 

\begin{lem}[Preservation of symmetry]\label{lem axisym1}
Suppose that $u\in C(\H)$ is bounded from below and symmetric with respect to $z$-axis.  
Then $S[u]$ given by \eqref{envelope op} is also symmetric with respect to $z$-axis. In particular, $\env$ satisfies the same symmetry condition as well. 
\end{lem}
\begin{proof}
Pick arbitrarily $p_0=(x_0, y_0, z_0)\in \H$ and $p_0'=(-x_0, -y_0, z_0)\in \H$.
By definition of $S[u]$, for any $\delta>0$ small, there exist $p_{i}=(x_{i}, y_{i}, z_{i})\in \H_{p_0}$ with $c_{i}\in [0, 1]$, $i=1, 2, 3$, such that
\begin{equation}\label{eq axisym1}
\sum_{i=1}^3 c_{i}=1, \quad \sum_{i=1}^3 c_{i}p_{i}=p_0
\end{equation}
and 
\begin{equation}\label{eq axisym2}
S[u](p_0)\geq \sum_{i=1}^3 c_{i} u(p_{i})-{\delta}. 
\end{equation}
Using the plane equation as in \eqref{eq plane}, we can express the relation $p_{i}=(x_{i}, y_{i}, z_{i})\in \H_{p_0}$ by 
\begin{equation}\label{eq axisym4}
y_0 x_{i}-x_0 y_{i}+2z_{i}-2z_0=0. 
\end{equation}
Set $p_{i}'=(-x_{i}, -y_{i}, z_{i})\in \H$. It is easily verified that $p_i'\in \H_{p_0'}$ for $i=1, 2, 3$. 

We can apply the symmetry of $u$ to \eqref{eq axisym2} to obtain that 
\begin{equation}\label{eq axisym3}
S[u](p_0)\geq \sum_{i=1}^3 c_{i} u(p_{i}')-\delta.
\end{equation}
By \eqref{eq axisym1}, it is also clear that 
\[
\sum_{i=1}^3 c_{i}p_{i}'=p_0'.
\]
It thus follows from \eqref{eq axisym3} that 
\[
S[u](p_0)\geq S[u](p_0')-\delta, 
\]
which implies that 
\[
S[u](p_0)\geq S[u](p_0')
\]
by letting $\delta\to 0$. Exchanging the roles of $p_0$ and $p_0'$, we obtain that $S[u](p_0)=S[u](p_0')$, which means that $S[u]$ is symmetric with respect to $z$-axis. 
As an immediate consequence of Theorem \ref{thm equivalence}, we can also deduce the symmetry of $\env$. 
\end{proof}

The additional symmetry assumption will largely facilitate our study on properties of the h-convex envelope later. A typical advantage with such symmetry is that the left and right invariant convexification actually coincide. 
\begin{thm}[Equivalence under symmetry]\label{thm axisym2}
Suppose that $u\in C(\H)$ is bounded from below and symmetric with respect to $z$-axis.  
Let $S[u]$ and $\St[u]$ be given as in \eqref{envelope op} and \eqref{envelope op right} respectively. Then $S[u]=\St[u]$ in $\H$. 
\end{thm}
\begin{proof}
 Fix $p_0=(x_0, y_0, z_0)\in \H$ arbitrarily. As in the proof of Lemma \ref{lem axisym1}, we can take for any $\delta>0$ small $p_{i}=(x_{i}, y_{i}, z_{i})\in \H_{p_0}$ and $c_{i}\in [0, 1]$ ($i=1, 2, 3$) such that \eqref{eq axisym1} and \eqref{eq axisym2} hold. We thus still have \eqref{eq axisym4}. 

Let us again take $p_{i}'=(-x_{i}, -y_{i}, z_{i})$. 
Then the symmetry of $u$ and \eqref{eq axisym2} yield that 
\begin{equation}\label{eq axisym5}
S[u](p_0)\geq \sum_{i=1}^3 c_{i} u(p_{i}')-\delta.
\end{equation}
Moreover, a direct calculation with the choices of $p_i'$ enables us to get
\[
y_0x_{i}'-x_0y_{i}'=-y_0x_{i}+x_0y_{i},
\]
which, in view of \eqref{eq axisym4}, implies that
\[
y_0x_{i}'-x_0y_{i}'+2z_{i}-2z_0=0.
\]
This amounts to saying that $p_{i}'\in \Ht_{p_0}$. 
Since we also have 
\[
\sum_{i=1}^3 c_i p_{i}'=p_0, 
\]
by \eqref{envelope op right} we obtain
\[
S[u](p_0)\geq \St[u](p_0)-\delta.
\]
Sending $\delta\to 0$, we are led to $S[u](p_0)\geq \St[u](p_0)$. As a parallel argument yields that $\St[u](p_0)\leq S[u](p_0)$, we complete the proof.
\end{proof}

As a result of Theorem \ref{thm axisym2}, we immediately obtain the following.
\begin{prop}\label{prop axisym3}
Suppose that $u\in C(\H)$ is bounded from below and symmetric with respect to $z$-axis. Then $u$ is h-convex if and only if it is right invariant h-convex. 
\end{prop}
The proof is based on the fact that $u=S[u]$ (resp., $u=\tilde{S}[u]$) if and only if $u$ is h-convex (resp. right invariant h-convex). 

\subsection{Examples of h-convex envelopes}\label{sec:example}

Let us give more concrete examples of the h-convex envelope and the operator $S$. We start with a simple example, for which $u$ can be convexified by the operator $S$ in one step.  

\begin{example}\label{ex:one step}
Let $f\in C(\R)$ be given by 
\[
f(t)=(|t|-1)^2, \quad t\in \R.
\]
Consider 
\[
u(p)=f\left(x^2+y^2+z^2\right), \quad  p=(x, y, z)\in \H. 
\]
It is clear that $u$ is coercive in $\H$.  One may guess that
\begin{equation}\label{eq ex1a}
\env(p)=\begin{cases} 0 & \text{if } x^2+y^2+z^2\leq 1,\\
u(p) &  \text{if } x^2+y^2+z^2>1.  
\end{cases}
\end{equation}
In other words, we expect that the h-convex envelope of $u$ is
\[
U(p)=F_E\left(x^2+y^2+z^2\right),\quad  p=(x, y, z)\in \H,
\]
where $F_E$ denotes the Euclidean convex envelope of $f$ in $\R$. 
In fact, this relation does hold. Note first that the right hand side of \eqref{eq ex1a} is h-convex in $\H$. 
Moreover, we have $S[u]=\env$ in $\H$ in this case. Indeed, since $u$ takes a minimum value $0$ and the minimizers of $u$ form a closed surface 
\[
x^2+y^2+z^2=1,
\]
for any $p=(x, y, z)\in \H$ such that $x^2+y^2+z^2<1$, the horizontal plane at $p$ and the closed surface must intersect at a closed curve on the plane. One therefore can take on the intersection three points whose convex combination coincides with $p$. This immediately yields that 
\[
S[u](p)=0
\]
for any $p=(x, y, z)\in \H$ such that $x^2+y^2+z^2<1$. It is clear that $S[u]=u$ at the rest of the points. 
\end{example}

\begin{example}\label{ex:two step}
We next give an example, showing that one sometimes needs to apply the operator $S$ twice to find the envelope $\env$.  Suppose that 
\[
u(p)=(z^2-1)^2, \quad p=(x, y, z)\in \H. 
\]
In this case, we have  
\begin{equation}\label{eq two step}
\env(x, y, z)=S^2[u](x, y, z)=\begin{cases} 0 & \text{if $|z|\leq 1$,}\\ (z^2-1)^2 &\text{if $|z|>1$.}
\end{cases}
\end{equation}
In fact, using the same argument for Example \ref{ex:one step}, we get
\[
S[u](x, y, z)= \begin{cases} 
0 & \text{if $|z|\leq 1$ and $(x, y)\neq (0, 0)$,}\\
(z^2-1)^2 & \text{if $|z|>1$ or $(x, y)=(0, 0)$.}
\end{cases} 
\]
One can easily apply the operator $S$ again on $S[u]$ to obtain the envelope given in \eqref{eq two step}.
\end{example}

\begin{example}\label{ex:failure with coercivity}
Let $u: \mathbb{R}^3\to \mathbb{R}$ be defined as
\[
u(x,y,z)=(x-y)z+(x-y)^2z^2+(x^2+y^2)^2+z^2.
\]
Clearly, $u$ is bounded from below and coercive. We show that $S[u](p)$ is not $h$-convex.

Note we have $u(p)\ge 0$ for all $p\in \mathbb{H}_0$ and
\[
S[u](0)=0.
\]
Let $t>0$ and ${h_t}=(t,t,0)\in \mathbb{H}_0$. Then the horizontal plane at ${h_t}$ is the collection of points 
\[
\mathbb{H}_{{h_t}}=(t,t,0)\cdot \mathbb{H}_0=\left\{\left(t+x, t+y, \frac{t}{2}(y-x)\right): x,y\in \mathbb{R}\right\}.
\]

Choose $c_1=c_2=\frac{1}{2}$ and $p_1, p_2\in \H_{{h_t}}$ defined as
\[
p_1=\left(t+t,t-t,\frac{t}{2}(-t-t)\right)=(2t, 0, -t^2)
\]
\[
p_2=\left(t-t,t+t,\frac{t}{2}(t+t)\right)=(0,2t, t^2).
\]
Then 
\[
u(p_1)=-2t^3+4t^6+16t^4+t^4,
\]
\[
u(p_2)=-2t^3+4t^6+16t^4+t^4,
\]
and 
\[
\begin{aligned}
S[u]({h_t})
&\le \frac{1}{2}u(p_1)+\frac{1}{2} u(p_2)\\
&=-2t^3+4t^6+17t^4.
\end{aligned}
\]

On the other hand, at $h_t^{-1}=(-t, -t, 0)$, we have 
\[
\begin{aligned}
0\le S[u](h_t^{-1})\le u(h_t^{-1})=4t^4.
\end{aligned}
\]

For $t>0$ sufficiently small, $S[u]({h_t})<0$. Thus,
\[
S[u]({h_t})+S[u](h_t^{-1})<0=2S[u](0).
\]
\end{example}

\section{Supersolution preserving property for elliptic equations}\label{sec:elliptic}


Let us now apply the notion of h-convex envelope to investigate the so-called supersolution preserving property. 
In this paper we focus our attention to the elliptic equation \eqref{elliptic eqn} but similar results can be shown for parabolic problems as well. 

Assume that $u$ is a coercive supersolution of \eqref{elliptic eqn}.  We aim to understand whether the h-convex envelope $\env$ is a supersolution of the same equation, since an affirmative answer, combined with a comparison principle, will imply that the unique solution is h-convex. 

This method is proposed in the Euclidean space by Alvarez, Lasry and Lions \cite{ALL}. One can use this method to show that the unique solution of the linear equation 
\[
u-\Delta u+\la\zeta, \nabla u \ra=f(p) \quad \text{in $\R^N$}
\]
is convex for any $\zeta\in \R^N$ provided that $f$ is convex in $\R^N$. We remark that in general we cannot expect that the same result holds in the Heisenberg group, as indicated by Example \ref{example no symmetry}.

\subsection{The right invariant envelope}\label{sec:right convex}

Our first result concerns the right convexification as given in \eqref{envelope op right}. In addition to (A1)(A2), we need the concavity condition on $F$ below.  
\begin{enumerate}
\item[(A3)] $(p, r, \xi, A)\mapsto F(p, r, \xi, A)$ is concave in the sense that for any $k\in \N$,
\[
\sum_{i=1}^k c_i F(p_i, r_i, \xi_i, A_i)\leq F\left(\sum_{i=1}^k c_ip_i,\ \sum_{i=1}^k c_i r_i,\ \sum_{i=1}^kc_i \xi_i,  \ \sum_{i=1}^k c_i A_i\right)
\]
holds for any $c_i\in [0, 1]$ with $\sum_i c_i =1$, $p_i\in \H$, $r_i\in \R$, $\xi_i\in \R^2$ and $A_i\in \S^2$ (for all $i=1, 2, \ldots, k$) satisfying 
\begin{equation}\label{F concavity}
p_i\in \tilde{\H}_{\sum_{i=1}^k c_i p_i}.
\end{equation}
\end{enumerate}

Let us define an operator of approximate right convexification as follows. For any $\vep>0$ small and any $p=(x, y, z)\in \H$, set
\begin{equation}\label{envelope app right}
\begin{aligned}
\tilde{S}_\vep[u](p):=\inf\bigg\{\sum_{i}c_i u(p_i)+{1\over \vep}\tilde{W}(p_1, p_2, p_3):\   c_i\in [0, 1], \ & p_i\in \H\  (i=1, 2, 3),\\
& \sum_i c_i=1, \ \sum_i c_i p_i=p \bigg\}.
\end{aligned}
\end{equation}
Here we denote
\[
\tilde{W}(p_1, p_2, p_3)=\sum_{i=1, 2, 3}c_i\tilde{g}_i^2\left(p_1, p_2, p_3\right),
\]
where
\[
\tilde{g}_i(p_1, p_2, p_3)=\left(\sum_i c_i y_i\right)x_i-\left(\sum_i c_ix_i\right) y_i-2z_i+2\sum_i c_i z_i.
\]
Note that, since the right invariant horizontal plane $\tilde{H}_p$ at $p\in \H$ is given by \eqref{eq plane right},  the quantity $\tilde{g}_i$ essentially measures how far the point $p_i$ is away from the right invariant horizontal plane passing through $\sum_i c_ip_i$. 

\begin{thm}[Supersolution preserving by right invariant convexification]\label{thm r-preserve}
Assume that (A1), (A2)  and (A3) hold. Let $u\in C(\H)$ be a supersolution of \eqref{elliptic eqn}. Suppose that $u$ satisfies the coercivity condition \eqref{coercive}. Let $\tilde{S}_\vep[u]$ be given by \eqref{envelope app right}. Then $\tilde{S}_\vep[u]$ is also a supersolution of \eqref{elliptic eqn} for any $\vep>0$ small. Moreover, $\tilde{S}[u]$ given by \eqref{envelope op right} is a lower semicontinuous supersolution of \eqref{elliptic eqn} as well. 
\end{thm}
\begin{proof}
Fix $\vep>0$ arbitrarily. Let us first show that $\tilde{S}_\vep[u]$ is a supersolution. Suppose that there is $\varphi\in C^2(\H)$ such that $\tilde{S}_\vep[u]-\varphi$ attains a minimum at $p^\vep=(x^\vep, y^\vep, z^\vep)\in \H$. We aim to show that 
\begin{equation}\label{super goal}
F\left(p^\vep, \tilde{S}_\vep[u](p^\vep), \nabla_H\varphi(p^\vep), (\nabla_H^2 \varphi)^\star(p_\vep)\right)\geq 0.
\end{equation}
We may further assume that $\varphi$ is bounded in $\H$. For our use later, we denote 
\[
\eta_\vep=\nabla\varphi(p^\vep), \quad Q_\vep=\nabla^2 \varphi(p^\vep).
\]

By definition of $\tilde{S}_\vep[u]$ and the coercivity of $u$, there exist $c_i\in [0, 1]$ and $p^\vep_i\in \H$ ($i=1, 2, 3$) with 
\begin{equation}\label{super comb1}
\sum_i c_i p^\vep_i=p^\vep
\end{equation}
such that
\begin{equation}\label{super comb2}
\tilde{S}_\vep[u](p^\vep)=\sum_i c_i u(p^\vep_i)+{1\over \vep}\tilde{W}(p^\vep_1, p^\vep_2, p^\vep_3).
\end{equation}
We may assume that $c_i\neq 0$ for every $i=1, 2, 3$, for otherwise we simply reduce to the situation with fewer terms in the sum above and the whole argument below still works. 

It is then clear that 
\begin{equation}\label{perturbed minimality}
\Phi_\vep(p_1, p_2, p_3)=\sum_{i=1, 2, 3} c_iu(p_i)-\varphi\left(\sum_{i}c_i p_i\right)+{1\over \vep}\tilde{W}(p_1, p_2, p_3)
\end{equation}
attains a minimum at $(p^\vep_1, p^\vep_2, p^\vep_3)$.

Denote $p_i^\vep=(x_i^\vep, y_i^\vep, z_i^\vep)$ for $i=1, 2, 3$. In view of the minimality of $\Phi_\vep$ at $(p_1^\vep, p_2^\vep, p_3^\vep)$, we use the Crandall-Ishii lemma \cite{CIL} to obtain, for any $\sigma>0$,  $(\eta_i, Q_i)\in \overline{J}_E^{2, -} u(p_i^\vep)$ satisfying 
\begin{equation}\label{jet first1}
c_i \eta_i=c_i \nabla \varphi(p^\vep)-{1\over \vep}\nabla_i \tilde{W}(p_1^\vep, p_2^\vep, p_3^\vep)
\end{equation}
and 
\begin{equation}\label{jet second1}
\mathbf{Q} \geq  \mathbf{A}-\sigma \mathbf{A}^2, 
\end{equation}
where 
\[
\mathbf{Q}=
\begin{pmatrix}
c_1 Q_1 &  0 & 0\\
0 & c_2 Q_2 & 0\\
0 & 0 & c_3 Q_3
\end{pmatrix}
\]
and
\begin{equation}\label{jet second add1}
\mathbf{A}=\begin{pmatrix}
c_1^2 Q_\vep & c_1c_2Q_\vep & c_1c_3 Q_\vep\\
c_1 c_2 Q_\vep & c_2^2 Q_\vep & c_2 c_3 Q_\vep\\
c_1 c_3 Q_\vep & c_2 c_3 Q_\vep & c_3^2 Q_\vep
\end{pmatrix}
-{1\over \vep} \nabla^2 \tilde{W}(p_1^\vep, p_2^\vep, p_3^\vep).
\end{equation}
For later use, we denote by $\mathbf{A}_1$ the first matrix on the right hand side of \eqref{jet second add1}. 

Let $p_i^\vep=(x_i, y_i, z_i)$ for $i=1, 2, 3$.  By direct calculations,  we have 
\[
\begin{aligned}
&{\partial \tilde{W}\over \partial {x_i}} =2c_i \left(\tilde{g}_i\sum_j c_j y_j-\sum_j c_j \tilde{g}_j y_j\right),\\
&{\partial \tilde{W} \over \partial {y_i}} =2c_i \left(-\tilde{g}_i\sum_j c_j x_j+\sum_j c_j \tilde{g}_j x_j\right),\\
&{\partial \tilde{W} \over \partial{z_i}}=4c_i\left(-\tilde{g}_i+\sum_j c_j\tilde{g}_j\right),
\end{aligned}
\]
and, moreover, 
\[
\begin{aligned}
{\partial^2 \tilde{W} \over \partial x_i \partial x_j}=& 2\delta_{ij} c_j  \left(\sum_j c_j y_j\right)^2-2c_i c_j \left(\sum_j c_j y_j\right)(y_i+y_j)+2c_i c_j \left(\sum_j c_j y_j^2\right),\\
{\partial^2\tilde{W} \over \partial y_i \partial y_j}=& 2\delta_{ij}c_j \left(\sum_j c_j x_j\right)^2-2c_i c_j \left(\sum_j c_j x_j\right)(x_i+x_j)+2c_ic_j \left(\sum_j c_j x_j^2\right),\\
{\partial^2 \tilde{W} \over \partial z_i\partial z_j} =& 8\delta_{ij}c_i-8c_ic_j,\\
\end{aligned}
\]
\[
\begin{aligned}
{\partial^2\tilde{W} \over \partial x_i \partial y_j} =& -2\delta_{ij}c_i \left(\sum_j c_j x_j\right)\left(\sum_j c_j y_j\right)\\
& +2c_i c_j x_i \left(\sum_j c_j y_j\right)+ 2c_ic_j y_j \left(\sum_jc_jx_j\right)-2c_i c_j \left(\sum_j c_jx_j y_j\right), \\
{\partial^2\tilde{W} \over \partial  x_i \partial z_j} =& -4\delta_{ij}c_i \left(\sum_j c_j y_j \right)+4c_ic_j y_j,\\
{\partial^2\tilde{W} \over \partial  y_i \partial z_j}=& 4\delta_{ij} c_i \left(\sum_j c_j x_j \right)-4c_ic_j x_j
\end{aligned}
\]
for all $i, j=1, 2, 3$. Then at the point $(p_1^\vep, p_2^\vep, p_3^\vep)$, we get
\[
\nabla_i \tilde{W}=2c_i\left(\tilde{g}_iy_\vep-\sum_j c_j \tilde{g}_jy_j, -\tilde{g}_i x_\vep+\sum_j c_j \tilde{g}_j x_j, -2\tilde{g}_i \right)
\]
and 
\[
\begin{aligned}
{\partial^2 \tilde{W} \over \partial x_i \partial x_j}=& 2\delta_{ij} c_j  y_\vep^2-2c_i c_j y_\vep(y_i+y_j)+2c_i c_j \left(\sum_j c_j y_j^2\right)\\
{\partial^2\tilde{W} \over \partial y_i \partial y_j}=& 2\delta_{ij}c_j x_\vep^2-2c_i c_j x_\vep(x_i+x_j)+2c_ic_j \left(\sum_j c_j x_j^2\right)\\
{\partial^2 \tilde{W} \over \partial z_i\partial z_j}=& 8\delta_{ij}c_i-8c_ic_j,\\
\end{aligned}
\]
\[
\begin{aligned}
{\partial^2\tilde{W} \over \partial x_i \partial y_j} =& -2\delta_{ij}c_i x_\vep y_\vep +2c_i c_j x_i y_\vep+ 2c_ic_j y_j x_\vep-2c_i c_j \left(\sum_j c_jx_j y_j\right), \\
{\partial^2\tilde{W} \over \partial  x_i \partial z_j} =& -4\delta_{ij}c_i y_\vep+4c_ic_j y_j,\\
{\partial^2\tilde{W} \over \partial  y_i \partial z_j}=&\  4\delta_{ij} c_i x_\vep-4c_ic_j x_j,
\end{aligned}
\]
where $\delta_{ij}$ is the Kronecker delta. 
Note that for any $p=(x, y, z)\in \H$, if $(\eta, Q)\in \overline{J}_E^{2, -}u(p)$, then $(\xi, P)\in \overline{J}_H^{2, -}u(p)$, where
\[
\xi =M_p \eta, \quad P=M_p Q M_p^T. 
\]
Here $M_p$ is a $2\times 3$ matrix given by
\[
M_p=\begin{pmatrix} 
1 & 0 & -y/2 \\
0 & 1 & x/2 
\end{pmatrix}
\]
and $M_p^T$ denotes its transpose. We use these relations to find $(\xi_i, P_i)\in \overline{J}_H^{2, -}u(p_i^\vep)$, that is, 
\begin{equation}\label{eh}
\xi_i=M_{p_i^\vep} \eta_i, \quad P_i=M_{p_i^\vep} Q_i M_{p_i^\vep}^T
\end{equation}
for all $i=1, 2, 3$. Then by \eqref{jet first1} and \eqref{eh}, we obtain
\begin{equation}\label{jet first2}
c_i\xi_i=c_i M_{p_i^\vep}  \nabla \varphi(p^\vep)-{1\over \vep} M_{p_i^\vep}\nabla_{i} \tilde{W}(p_1^\vep, p_2^\vep, p_3^\vep),
\end{equation}
which, by direct calculations, yields that
\begin{equation}\label{jet first}
\sum_i c_i\xi_i=\nabla_H \varphi(p^\vep),
\end{equation}
since 
\[
\begin{aligned}
&M_{p_i^\vep}\nabla_{i} \tilde{W}(p_1^\vep, p_2^\vep, p_3^\vep)\\
&=2c_i\left( \tilde{g}_i \sum_i c_i y_i-\sum_i c_i \tilde{g}_i y_i+\tilde{g}_i y_i, \  -\tilde{g}_i \sum_i c_i x_i+\sum_i c_i \tilde{g}_i x_i -\tilde{g}_ix_i\right)^T.
\end{aligned}
\]
More calculations are needed for the second derivatives. For any $(a, b)\in \R^2$, set
\[
\ell[a, b]=\left(a, b, -{y_1\over 2}a+{x_1\over 2}b, a, b, -{y_2\over 2}a+{x_2\over 2}b, a, b, -{y_3\over 2}a+{x_3\over 2}b\right)^T\in \R^9.
\]
We multiply both sides of \eqref{jet second1} by $\ell[a, b]$ from the left and by its transpose $\ell[a, b]^T$ from the right. We first have 
\begin{equation}\label{jet second2}
\la \mathbf{Q} \ell[a, b], \ell[a, b]\ra
=\la \mathbf{P} k[a, b], k[a, b]\ra \\
= \la \left(\sum_i  c_i P_i\right)(a, b)^T, (a, b)^T\ra,
\end{equation}
where $k[a, b]=(a, b, a, b, a, b)^T\in \R^6$ and 
\[
\mathbf{P}=\begin{pmatrix}
c_1 P_1 &  0 & 0\\
0 & c_2 P_2 & 0\\
0 & 0 & c_3 P_3
\end{pmatrix}.
\]

 For the right hand side, straightforward calculations yield
\begin{equation}\label{jet second3}
\begin{aligned}
\la \mathbf{A}_1 \ell[a, b], \ell[a, b]\ra
& =(a, b)\left(\sum_i c_i M_{p_i^\vep}\right) Q_\vep \left(\sum_i c_i M_{p_i^\vep}^T \right) \begin{pmatrix}a\\ b\end{pmatrix}\\
&= \la \left(\nabla^2_H \varphi\right)^\star (p^\vep) (a, b)^T, (a, b)^T\ra.
\end{aligned}
\end{equation}
In addition, we can calculate to see that
\begin{equation}\label{key second}
\nabla^2 \tilde{W}(p_1^\vep, p_2^\vep, p_3^\vep)\ell[a, b]=0, 
\end{equation}
which yields 
\begin{equation}\label{jet second4}
\la\left( \nabla^2 \tilde{W}(p_1^\vep, p_2^\vep, p_3^\vep)\right) \ell[a, b], \ell[a, b]\ra=0
\end{equation}
and 
\begin{equation}\label{jet second5}
\la\left( \nabla^2 \tilde{W}(p_1^\vep, p_2^\vep, p_3^\vep)\right)^2 \ell[a, b], \ell[a, b]\ra=0.
\end{equation}
Combining \eqref{jet second2}--\eqref{jet second5} with \eqref{jet second1}, we are led to 
\begin{equation}\label{jet second}
\sum_i c_i P_i\geq (\nabla_H^2 \varphi)^\star(p^\vep)-O(\sigma)
\end{equation}
when $\sigma>0$ is taken small. 

Since $u$ is a supersolution of \eqref{elliptic eqn}, we get
\[
F\left(p_i^\vep, u(p_i^\vep), \xi_i, P_i\right)\geq 0\quad i=1, 2, 3.
\]
Multiplying the inequality above by $c_i$ and summing them up, we deduce that 
\[
\sum_i c_i F\left(p_i^\vep, u(p_i^\vep), \xi_i, P_i\right)\geq 0. 
\]
By the concavity condition (A3), we obtain 
\[
F\left(\sum_i c_i p_i^\vep,\ \sum_i u(p_i^\vep),\ \sum_i c_i \xi_i,\ \sum_i c_iP_i\right)\geq 0.
\]
Note that \eqref{super comb2} implies that 
\begin{equation}\label{super comb3}
\tilde{S}_\vep[u](p^\vep)\geq \sum_i c_i u(p_i^\vep).
\end{equation}
Adopting \eqref{super comb1}, \eqref{jet first},  \eqref{jet second} and \eqref{super comb3} as well as (A1) and (A2), we end up with 
\[
F\left(p^\vep, \tilde{S}_\vep[u](p^\vep), \nabla_H\varphi(p^\vep), (\nabla_H^2 \varphi)^\star(p_\vep)-O(\sigma)\right)\geq 0,
\]
which yields \eqref{super goal} as desired by letting $\sigma\to 0$. 

We finally prove that $\tilde{S}[u]$ is a supersolution. The proof is essentially based on the standard stability theory.  Suppose that there exist $\hat{p}\in \H$ and $\varphi\in C^2(\H)$ such that $\tilde{S}[u]-\varphi$ attains a strict minimum in $\H$ at $\hat{p}$. Since $u$ is coercive,  there exist $\hat{p}_i\in \H_{\hat{p}}$ and $c_i\in [0, 1]$ ($i=1, 2, 3$) satisfying 
\[
\sum_i c_i=1, \quad \sum_i c_i \hat{p}_i=\hat{p}
\] 
such that 
\[
(p_1, p_2, p_3)\mapsto \sum_i c_iu(p_i)-\varphi\left(\sum_i c_i p_i\right)
\]
attains a minimum at $(\hat{p}_1, \hat{p}_2, \hat{p}_3)$.  It follows that for any $\vep>0$ there exist $p_i^\vep\in \H$ such that \eqref{perturbed minimality} attains a minimum at $(p_1^\vep, p_2^\vep, p_3^\vep)$. Let $p^\vep=\sum_i c_i p_i^\vep$. Then by definition \eqref{perturbed minimality} amounts to saying that $\tilde{S}_\vep[u]-\varphi$ attains a minimum in $\H$ at $p^\vep$. 

Moreover, we claim that $p^\vep\to \hat{p}$ as $\vep\to 0$. Indeed, using the coercivity of $u$, we may take a subsequence so that as $\vep\to 0$, $p_i^\vep\to q_i$ for some $q_i\in \H$  ($i=1, 2, 3$) and $p^\vep\to q_0=\sum_i c_i q_i$. Since  
\begin{equation}\label{perturb comparison}
\Phi_\vep(p_1^\vep, p_2^\vep, p_3^\vep)\leq \Phi_\vep(\hat{p}_1, \hat{p}_2, \hat{p}_3), 
\end{equation}
we deduce that 
\[
{1\over \vep}\tilde{W}(p_1^\vep, p_2^\vep, p_3^\vep)\leq \sum_i c_iu(\hat{p}_i) -\sum_i c_i u(p_i^\vep)-\varphi(\hat{p})+\varphi(p^\vep).
\]
Noticing that the right hand side is bounded above, we get, as $\vep\to 0$, 
\[
\tilde{W}(p_1^\vep, p_2^\vep, p_3^\vep)\to 0, 
\]
which yields that $q_i\in \H_{q_0}$. Hence, letting $\vep\to 0$ in \eqref{perturb comparison}, we obtain 
\[
\tilde{S}[u](q_0)-\varphi(q_0)\leq \tilde{S}[u](\hat{p})-\varphi(\hat{p}),
\]
which implies that $q_0=\hat{p}$ due to the strict minimum of $\tilde{S}[u]-\varphi$ at $\hat{p}$. We complete the proof of the claim. 

Since we have shown that $\tilde{S}_\vep[u]$ is a supersolution of \eqref{elliptic eqn}, the inequality \eqref{super goal} holds. Noticing that $\tilde{S}[u]\geq \tilde{S}_\vep[u]$, by (A2) we get
\[
F\left(p^\vep, \tilde{S}[u](p^\vep), \nabla_H\varphi(p^\vep), (\nabla_H^2 \varphi)^\star(p_\vep)\right)\geq 0.
\]
Passing to the limit as $\vep\to 0$, we apply the continuity of $F$ and the fact that $p_\vep\to \hat{p}$ to obtain
\[
F\left(\hat{p},\ \tilde{S}[u](\hat{p}),\ \nabla_H\varphi(\hat{p}),\ (\nabla_H^2 \varphi)^\star(\hat{p}) \right)\geq 0, 
\]
as desired. 
\end{proof}

\begin{rmk}
The key to the proof of Theorem \ref{thm r-preserve} lies at the relations \eqref{jet first} and \eqref{key second}. One may wonder whether one can use the same method to show $S[u]$ is a supersolution by 
replacing $\tilde{W}$ with $W$ given by
\[
W(p_1, p_2, p_3)=\sum_i c_i g_i^2(p_1, p_2, p_3), \quad \text{$p_i\in \H$, i=1, 2, 3}, 
\]
where 
\[
g_i(p_1, p_2, p_3)=\left(\sum_i c_i y_i\right) x_i -\left(\sum_i c_i x_i \right) y_i +2z_i -2\sum_i c_i z_i
\]
for $p_i=(x_i, y_i, z_i)$, $i=1, 2, 3$. We however are not able to obtain the desired conclusion in this case. Note that the horizontal gradient with respect to the variable $p_i$ is calculated to be 
\[
\begin{aligned}
&\nabla_{H, i} W\\
&=\left(2c_i g_i \sum_i c_iy_i- 2c_ig_iy_i -2c_i \sum_i c_i g_i y_i, -2c_i g_i \sum_i c_i x_i +2c_ig_i x_i +2c_i \sum_i c_i  g_ix_i \right).
\end{aligned}
\]
We therefore have
\[
\sum_i \nabla_{H, i}W=4\left(- \sum_i c_i g_i y_i,\  \sum_i c_ig_ix_i\right),
\]
which fails to vanish in general. In other words, one cannot obtain \eqref{jet first} in this case. Similarly, \eqref{jet second} cannot be expected either. 
\end{rmk}

\begin{rmk}
It is implied by (A3) that $p\mapsto -F(p, r, \xi, A)$ is right invariant h-convex. This condition is necessary for the result in Theorem \ref{thm r-preserve}. In fact, in the trivial case when 
\[
F(p, r, \xi, A)=r-f(p)  
\]
for $(p, r, \xi, A)\in \H\times \R\times \R^2\times \S^2$, $u=f$ is clearly the unique solution to this entirely degenerate equation. The assumption (A3) reduces to the right invariant h-convexity of $f$, which certainly implies that $\tilde{S}[u]=u=f$ in $\H$. 
\end{rmk}

\subsection{The left invariant envelope with symmetry}

Under the symmetry of $u$ with respect to $z$-axis, Theorem \ref{thm r-preserve} enables us to show that $\env$ is a supersolution of \eqref{elliptic eqn} if $u$ itself is a supersolution. 

\begin{thm}[Symmetric supersolution preserving]\label{thm preserve rot}
Assume that (A1)--(A3) hold. Let $u$ be a lower semicontinuous supersolution of \eqref{elliptic eqn}. Suppose that $u$ satisfies the coercivity condition \eqref{coercive}. Assume in addition that $u$ is symmetric with respect to $z$-axis.  Let $S[u]$ be given by \eqref{envelope op}. Then $S[u]$ is also a supersolution of \eqref{elliptic eqn}. Moreover, the convex envelope $\env$ is a supersolution of \eqref{elliptic eqn} as well. 
\end{thm}
\begin{proof}
By Theorem \ref{thm r-preserve}, we see that $\St[u]$ is a supersolution of \eqref{elliptic eqn}. Due to the symmetry of $u$, by Theorem \ref{thm axisym2} we have $\St[u]=S[u]$ in $\H$, which implies that $S[u]$ is also a supersolution. Moreover, $S[u]$ is also symmetric thanks to Lemma \ref{lem axisym1}. In addition, by Lemma \ref{lem coercive}, $S[u]$ also satisfies the coercivity condition \eqref{coercive}. 

We thus can iterate the argument to show that $S^n[u]$ is a supersolution of \eqref{elliptic eqn} for any $n=1, 2, \ldots$ Since $S^n[u]\to \env$ in $\H$ locally uniformly as $n\to \infty$ (see Remark \ref{rmk uniform convergence}), we conclude that $\env$ is a supersolution by the standard stability theory of viscosity solutions. 
\end{proof}

The convexity preserving property of the h-convex envelope enables us to show the h-convexity of a symmetric solution of \eqref{elliptic eqn} provided that a comparison principle for \eqref{elliptic eqn} is available. 

\begin{cor}[H-convexity of symmetric solutions]\label{cor axisym1}
Assume that (A1)--(A3) hold. Let $u$ be a continuous solution of \eqref{elliptic eqn} satisfying \eqref{coercive}. Assume in addition that $u$ is symmetric with respect to $z$-axis. If the comparison principle for \eqref{elliptic eqn} holds, then $u$ is h-convex in $\H$.  
\end{cor}
\begin{proof}
Since $\env$ is a supersolution by Theorem \ref{thm preserve rot},  we apply the comparison principle to deduce that $u\leq \env$. Noticing that $\env\leq u$ by definition, we conclude that $u=\env$ in $\H$, which yields the h-convexity of $u$. 
\end{proof}
In the proof above, it in fact suffices to show that $u=S[u]$ hold in $\H$ by using the same argument.

It is natural to ask when the solution $u$ is symmetric with respect to $z$-axis. It turns out to be sufficient to have the following two ingredients at hand.

\begin{itemize}
\item A comparison principle that allows coercive solutions as mentioned in Remark \ref{rmk comparison} is needed.
\item The following symmetric assumption (A4) on $F$ is additionally imposed. 
\end{itemize}
\begin{enumerate}
\item[(A4)] $(p, \xi)\mapsto F\left(p, r, \xi, A\right)$ is symmetric in the sense that 
\[
F(p, r, \xi, A)=F(p', r, -\xi, A)
\]
for any $p=(x, y, z)\in \H$ with $p'=(-x, -y, z)\in \H$, $r\in \R$, $\xi\in \R^2$ and $A\in \S^2$. 
\end{enumerate}

\begin{prop}[Symmetry of solutions]\label{prop axisym}
Assume that (A4) holds. If $u$ is a subsolution (resp., supersolution, solutions) of \eqref{elliptic eqn}, then $v(x, y, z)=u(-x, -y, z)$ is also a subsolution (resp., supersolution, solutions)  of \eqref{elliptic eqn}. In particular, if $u$ is the unique solution of \eqref{elliptic eqn}, then $u$ is symmetric with respect to $z$-axis. 
\end{prop}
\begin{proof}
Let us only verify the subsolution part. Suppose that there exist $p_0=(x_0, y_0, z_0)\in \H$ and $\phi\in C^2(\H)$ such that $v-\phi$ attains a maximum at $p_0$. We aim to show that 
\begin{equation}\label{axisym sub1}
F(p_0, v(p_0), \nabla_H \phi(p_0), (\nabla_H^2 \phi)^\star(p_0))\leq 0. 
\end{equation}
It is clear that $u-\psi$ attains a maximum at $p_0'=(-x_0, -y_0, z_0)$, where $\psi(x, y, z)=\phi(-x, -y, z)$.
Since $u$ is a subsolution of \eqref{elliptic eqn}, we have
\begin{equation}\label{axisym sub2}
F(p_0', u(p_0'), \nabla_H \psi(p_0'), (\nabla_H^2\psi)^\star(p_0'))\leq 0.
\end{equation}
In order to use \eqref{axisym sub2} to show \eqref{axisym sub1}, we make the following straightforward calculations: for any $(x, y, z)\in\H$,
\[
(X \psi)(x, y, z)=\left(-{\partial\phi \over \partial x}  -{y\over 2}{\partial\phi\over \partial z} \right)(-x, -y, z)=-(X\phi)(-x, -y, z);
\]
\[
(Y \psi)(x, y, z)=\left(-{\partial\phi \over \partial y} +{x\over 2}{\partial\phi \over  \partial z} \right)(-x, -y, z)=-(Y\phi)(-x, -y, z);
\]
\[
(X^2\psi)(x, y, z)=\left({\partial^2 \phi\over \partial x^2}+y{\partial^2\phi \over \partial x\partial z}  +{y^2\over 4}{\partial^2 \phi\over \partial z^2 }\right)(-x, -y, z)=(X^2 \phi)(-x, -y, z);
\]
\[
(Y^2\psi)(x, y, z)=\left({\partial^2 \phi \over \partial y^2} -x{\partial^2\phi \over \partial y\partial z} +{x^2\over 4}{\partial^2 \phi\over \partial z^2 }\right)(-x, -y, z)=(Y^2 \phi)(-x, -y, z);
\]
\[
\begin{aligned}
{1\over 2}\left(X Y\psi+YX \psi\right)(x, y, z)&=\left({\partial^2\phi \over \partial x\partial y}-{x\over 2} {\partial^2\phi \over \partial x\partial z}+{y\over 2} {\partial^2\phi \over \partial y\partial z}-{xy\over 4} {\partial^2 \phi\over \partial z^2 }\right)(-x, -y, z)\\
&={1\over 2}\left(X Y\phi+YX \phi\right)(-x, -y, z).
\end{aligned}
\]
It then follows that 
\[
\nabla_H\psi(p_0')=-\nabla_H \phi(p_0), \quad (\nabla_H^2\psi)^\star(p_0')=(\nabla_H^2\phi)^\star(p_0). 
\]
Plugging these into \eqref{axisym sub2}, we are led to 
\[
F(p_0', v(p_0), -\nabla_H \phi(p_0), (\nabla^2_H \phi)^\star(p_0))\leq 0. 
\]
By the symmetry condition (A4), we get \eqref{axisym sub1} immediately. 
\end{proof}

Theorem \ref{thm preserve rot} and Proposition \ref{prop axisym} thus imply the following. 

\begin{thm}[H-convexity of solutions]\label{cor h-convex}
Assume that (A1)--(A4) hold. Assume that the comparison principle for \eqref{elliptic eqn} holds. Let $u$ be the unique continuous solution of \eqref{elliptic eqn} satisfying \eqref{coercive}. Let $S[u]$ be given by \eqref{envelope op}.  Then $S[u]=u$ in $\H$. 
In particular, $u$ is h-convex in $\H$.  
\end{thm}
\begin{rmk}\label{rmk concave sym}
As pointed out in Proposition \ref{prop axisym3}, under (A4),  one can replace \eqref{F concavity} in the concavity assumption (A3) by 
\[
p_i\in \H_{\sum_{i=1}^k c_i p_i}.
\]
In other words, we can assume that $p\mapsto F(p, r, \xi, X)$ satifies h-concavity instead of the right invariant h-concavity. 
\end{rmk}

\subsection{Examples in a special case}\label{sec:special eqn}

Since a comparison theorem is available in \cite{I5} when $F$ is of the form \eqref{special form}, our convexity result in particular applies to the semilinear equation \eqref{special eqn}, as shown in Theorem \ref{intro thm}. 
Note that the symmetry condition of $f$ enables us to assume h-convexity of $f$ rather than its right invariant h-convexity, since they are equivalent, as mentioned in Proposition \ref{prop axisym3} and Remark \ref{rmk concave sym}. 

We here skip discussion on existence of solutions to \eqref{special eqn} and the more general equation \eqref{elliptic eqn}, since it can be obtained by Perron's method. But we provide a concrete example below in order to avoid possible triviality of our results. 

\begin{example}\label{example h-convex sol1}
Consider \eqref{special eqn} with $0\leq \alpha<1/3$, $\beta=0$ and 
\[
f(x, y, z)=(1-3\alpha)(x^2+y^2)+x^2y^2+2z^2-4\alpha
\]
for $(x, y, z)\in \H$. It is clear that $f$ is coercive and symmetric with respect to $z$-axis. Moreover, $f$ is h-convex in $\H$, since 
\begin{equation}\label{nontrivial1}
(\nabla_H^2 f)^\star(x, y, z)=\begin{pmatrix}
2(1-3\alpha)+3y^2 & 3xy \\ 3xy & 2(1-3\alpha)+3x^2
\end{pmatrix}
\end{equation}
is nonnegative for all $(x, y, z)\in \H$. We remark that $f$ is not convex in the Euclidean sense. 

By direct calculations, one can verify that the unique solution in this case is 
\[
u(x, y, z)=x^2+y^2+x^2y^2+2z^2. 
\]
for  $(x, y, z)\in \H$. It is also easily seen that $u$ is h-convex in $\H$ (but not convex in $\R^3$), since it happens to coincide with $f$ with $\alpha=0$. 
\end{example}

Although in Theorem \ref{cor h-convex} $u$ is assumed to be coercive in the sense of \eqref{coercive}, it is worth stressing that for \eqref{special eqn} (with $\alpha\geq 0, \beta\geq 0$), no coercivity conditions on $u$ or on $f$ are essentially needed. The coercivity of $f$, together with its h-convexity, does assist one to show the coercivity of the solution $u$ by the comparison theorem; in fact, $u=f$ is clearly a subsolution of \eqref{special eqn}. 

On the other hand,  if $u$ is not coercive, one can replace $f$ by, for example,  
\[
f_\vep(x, y, z)=f(x, y, z)+\vep(x^2+y^2+z^2)
\]
with $\vep>0$ small so that the corresponding solution $u^\vep$ is also coercive. Since $f$ is still h-convex and symmetric, we can then apply Theorem \ref{cor h-convex} to deduce the h-convexity of $u^\vep$. Letting $\vep\to 0$, we have $u^\vep\to u$ locally uniformly by adopting the standard stability theory of viscosity solutions. The h-convexity of $u$ follows immediately. 

As Theorem \ref{cor h-convex} holds  for \eqref{special eqn} without coercivity assumptions on $u$, the following trivial example is thus covered. 

\begin{example}\label{example h-convex sol2}
Pick $\alpha, \beta\geq 0$ arbitrarily. Let $f\equiv C$ in $\H$ for any $C\in \R$. It is then clear that $u=C$ is the unique solution of \eqref{special eqn}, which is obviously h-convex.  
\end{example}

For the equation \eqref{special eqn}, (A4) is satisfied if $f$ is symmetric with respect to the $z$-axis. On the other hand, if the symmetry condition on $f$  is dropped, then in general the unique solution $u$ of \eqref{special eqn} will not be symmetric and the h-convexity of $u$ may fail to hold, as shown in Example \ref{example no symmetry}. We give another example below based on slight modification of Example \ref{example no symmetry}. 

\begin{example}\label{example no symmetry2}
Let us consider \eqref{special eqn} with $\alpha=\beta=1$ $\mathcal{A}=\{\pm(0, 2)\}$; in other words, we study the following equation: 
\begin{equation}\label{example no symmetry eq2}
u=\Delta_H u+\left|\la (0, 2), \nabla_H u\ra\right|+f(p)\quad \text{in $\H$.} 
\end{equation}
Here we take 
\[
f(x, y, z)=2xz+x^2y+{1\over 4}x^4+{3\over 2}y^2+6|y|-1.
\]
It is clear that $f$ is not symmetric with respect to $z$-axis.  Based on similar calculations for Example \ref{example no symmetry}, we can verify that $u$ given in \eqref{example no symmetry sol} is the unique solution of \eqref{example no symmetry eq2}, which is not h-convex around the origin. However, the function $f$ is still h-convex in $\H$. 
\end{example}

Besides the symmetry condition, we also have the following example suggesting that the concavity $\xi\mapsto F(p, r, \xi, A)$ seems to be necessary. 
\begin{example}\label{ex strong concavity}
Consider the equation
\[
u+|\nabla_H u|^2=f(p)
\]
with $f(x,y,z)=(4\vep^2+1)(x^2+y^2)+2z$ being h-convex and right invariant h-convex. However, it is easily seen that 
\[
u(x, y, z)=-\vep x^2-\vep y^2+2z
\]
is a solution but it is neither h-convex nor right invariant h-convex. The drawback of this example is that $u$ is not coercive. It would be interesting to construct better examples satisfying coercivity. 
\end{example}

\subsection{The Euclidean envelope}\label{sec:euclidean}

We conclude this part by providing a result on the Euclidean convexity of solutions to \eqref{elliptic eqn}, although it is not our main concern in this paper. 
In this case, we need to  strengthen the assumption (A3):
\begin{enumerate}
\item[(A3)'] $(p, r, \xi, A)\mapsto F(p, r, \xi, A)$ is concave in the Euclidean sense, namely, for any $k\in \N$,
\[
\sum_{i=1}^k c_i F(p_i, r_i, \xi_i, A_i)\leq F\left(\sum_{i=1}^k c_ip_i,\ \sum_{i=1}^k c_i r_i,\ \sum_{i=1}^kc_i \xi_i,  \ \sum_{i=1}^k c_i A_i\right)
\]
holds for any $c_i\in [0, 1]$, $p_i\in \H$, $r_i\in \R$, $\xi_i\in \R^2$ and $A_i\in \S^2$ ($i=1, 2, \ldots, k$).
\end{enumerate}

\begin{thm}[Supersolution preserving of the Euclidean envelope]\label{thm e-convex}
Assume that (A1), (A2) and (A3)' hold. Let $u$ be a lower semicontinuous supersolution of \eqref{elliptic eqn}. Suppose that $u$ satisfies the coercivity condition \eqref{coercive}. Let $\Gamma_E u$ be the Euclidean convex envelope of $u$ given by \eqref{euclidean envelope2} with $n=3$. Then $\Gamma_E u$ is also a supersolution of \eqref{elliptic eqn}. Moreover, if the comparison principle for \eqref{elliptic eqn} holds, then the unique solution of \eqref{elliptic eqn}  is convex in $\R^3$ (in the Euclidean sense). 
\end{thm}
\begin{proof}
The proof is based on slight modification of that of Theorem \ref{thm r-preserve}. Since we consider the Euclidean convex combination, we can simply set the penalty term ${\tilde{W}}\equiv 0$ in the proof of Theorem \ref{thm r-preserve} and the whole argument still works. Indeed, if we let $\varphi\in C^2(\H)$ denote the test function of $\Gamma_E u$ at a point $p_0\in \H$ and let $(\hat{p}_1, \hat{p}_2, \hat{p}_3)$ denote the minimizer of 
\[
(p_1, p_2, p_3)\mapsto \sum_{i=1, 2, 3} c_iu(p_i)-\varphi\left(\sum_{i}c_i p_i\right),
\]
then for any $\sigma>0$ we have $(\eta_i, Q_i)\in \overline{J}^{2, -} u(p_i)$ such that 
\[
\eta_i=\nabla \varphi(p_0)
\]
and 
\[
\begin{pmatrix}
c_1 Q_1 &  0 & 0\\
0 & c_2 Q_2 & 0\\
0 & 0 & c_3 Q_3
\end{pmatrix}
\geq 
\begin{pmatrix}
c_1^2 Q_0 & c_1c_2Q_0 & c_1c_3 Q_0\\
c_1 c_2 Q_0 & c_2^2 Q_0 & c_2 c_3 Q_0\\
c_1 c_3 Q_0 & c_2 c_3 Q_0 & c_3^2 Q_0
\end{pmatrix}
-\sigma 
\begin{pmatrix}
c_1^2 Q_0 & c_1c_2Q_0 & c_1c_3 Q_0\\
c_1 c_2 Q_0 & c_2^2 Q_0 & c_2 c_3 Q_0\\
c_1 c_3 Q_0 & c_2 c_3 Q_0 & c_3^2 Q_0
\end{pmatrix}^2,
\]
where $Q_0=\nabla^2 \varphi(p_0)$. We next turn to the horizontal jets $(\xi_i, P_i)\in \overline{J}_H^{2, -}u(p_i)$. By the calculations similar to the proof of Theorem \ref{thm r-preserve}, we get the following variants of \eqref{jet first} and \eqref{jet second}:
\[
\sum_i c_i \xi_i=\nabla_H \varphi(p_0);
\]
\[
\sum_i c_i P_i\geq (\nabla_H^2\varphi)^\star(p_0)-O(\sigma). 
\]
Applying the supersolution property of $u$ at $\hat{p}_i$ and taking the weighted average, we obtain 
\[
\sum_i c_i F\left(\hat{p}_i, u(\hat{p}_i), \xi_i, P_i\right)\geq 0,
\]
which, by (A3)' and the ellipticity of $F$,  yields 
\[
F\left(p_0, \Gamma_E u(p_0), \nabla_H\varphi(p_0), (\nabla_H^2 \varphi)^\star(p_0)-O(\sigma)\right)\geq 0.
\]
We complete the proof by letting $\sigma\to 0$. 

If the comparison principle for \eqref{elliptic eqn} holds, then the unique solution $u$ satisfies $u\leq \Gamma_E u$. Since the reverse inequality clearly holds, the convexity of $u$ follows immediately. 
\end{proof}

This result is closely related to the vast existing literature on the convexity of solutions to  a large variety of elliptic and parabolic equations in the Euclidean space; see \cite{Kor, K4, Ke, ES1, GGIS, ALL, Ju, LSZ} etc. 
In fact, Theorem \ref{thm e-convex} gives a clearer and more direct answer for equations like \eqref{elliptic eqn}.


As the Euclidean convexity implies h-convexity, Theorem \ref{thm e-convex} helps us understand h-convexity of solutions to a large class of nonlinear elliptic equations including \eqref{special eqn}. A trivial example for Theorem \ref{thm e-convex} is in Example \ref{example h-convex sol2}. We give a less trivial one below. 

\begin{example}
Let us revisit \eqref{special eqn} with $\alpha=1$, $\beta=0$ and 
\[
f(x, y, z)=x+2z^2+\vep(x^2+y^2)
\]
for $\vep>0$ and $(x, y, z)\in \H$. We are unable to apply Theorem \ref{cor h-convex} to obtain $h$-convexity of the unique solution $u$, since $f$ is not symmetric with respect to $z$-axis and therefore (A4) fails to hold. However, we can use Theorem \ref{thm e-convex} to conclude that $u$ is convex in the Euclidean sense. In fact, the unique solution in this case can be explicitly written as 
\[
u(x, y, z)=(1+\vep)(x^2+y^2+4)+x+2z^2
\] 
for $(x, y, z)\in \H$. 

On the other hand, when $f$ is not convex but only h-convex like Example \ref{example h-convex sol1}, we can only show the h-convexity of $u$ under the symmetry condition of $f$. 

It is worth stressing that even if we want prove the Euclidean convexity of solutions to sub-elliptic PDEs, we probably still need the strong concavity of $F$ in all arguments, as shown in Example \ref{ex strong concavity}. 
\end{example}

\end{document}